\documentclass[11pt,a4paper,reqno]{amsart} 
\usepackage{anysize}
\marginsize{3.5cm}{3.5cm}{2.2cm}{2.2cm}

\usepackage{hyperref} 
\hypersetup{
    colorlinks=true,       
    linkcolor=red,          
    citecolor=blue,        
    filecolor=magenta,      
    urlcolor=cyan           
}
\usepackage{standalone}
\usepackage[english]{babel}
\usepackage{amsmath}
\usepackage{amsfonts,amssymb}
\usepackage{enumerate}
\usepackage{mathrsfs}
\usepackage{amsthm}
\usepackage{tikz}
\usetikzlibrary{intersections,decorations.markings}
 \usetikzlibrary{patterns}

\newfont{\bb}{msbm10 at 12pt}
\newfont{\bbt}{msbm10 at 9pt}
\def\r{\mathbb R}
\def\d{\mathbb D}
\def\n{\mathbb N}
\def\z{\mathbb Z}
\def\t{\mathbb T}

\def\b{\hbox{ B}}
\def\fr{\mathscr F}
\def\sr{\mathscr S}
\def\dr{\mathscr D}
\def\w{\mathscr W}
\def\er{\mathscr E}

\DeclareMathOperator{\deter}{det}
\DeclareMathOperator{\supp}{supp}
\DeclareMathOperator{\D}{\langle D \rangle}

\DeclareMathOperator{\singsupp}{sing \ supp}

\DeclareMathOperator{\op}{Op}

\newcommand{\norm}[1]{\left\Vert #1 \right\Vert}
\newcommand{\abs}[1]{\left\vert #1 \right\vert}
\newcommand{\set}[1]{\left\{#1\right\}}
\newcommand{\eps}{\epsilon}

\usepackage{hyperref}
\hypersetup{
    colorlinks=true,
    linkcolor=red,
    filecolor=red,      
    urlcolor=blue,
    citecolor=red,
    pdftitle={paracomposition },
    bookmarks=true,
    pdfpagemode=FullScreen,
}

\numberwithin{equation} {section}

\newcommand{\keyword}[1]{\textbf{\textit{Keywords---}} \textbf{#1}}

\theoremstyle{plain}\newtheorem{lemma}{Lemma}[section]
\theoremstyle{plain}\newtheorem{proposition}{Proposition}[section]
\theoremstyle{plain}\newtheorem{theorem}{Theorem}[section]
\theoremstyle{plain}\newtheorem{definition}{Definition}[section]
\theoremstyle{plain}\newtheorem{remark}{Remark}[section]
\theoremstyle{plain}\newtheorem{corollary}{Corollary}[section]
\theoremstyle{plain}
\theoremstyle{plain}\newtheorem{notation}{Notation}[section]
\theoremstyle{plain}\newtheorem{definition-proposition}{Definition-Proposition}[section]
\numberwithin{equation}{section}

\title{On paracomposition and change of variables in Paradifferential operators}
\author{Ayman Rimah Said}

\def\notina[#1]#2{\begingroup\def\thefootnote{\fnsymbol{footnote}}\footnote[#1]{#2}\endgroup}

\begin{document}


\begin{abstract}
In this paper we revisit the hypothesis needed to define the ``paracomposition" operator, an analogue to the classic pull-back operation in the low regularity setting, first introduced by S. Alinhac in \cite{Alinhac86}. More precisely we do so in two directions. First we drop the diffeomorphism hypothesis. Secondly we give estimates in global Sobolev and Zygmund spaces. Thus we fully generalize Bony's classic paralinearasition theorem giving sharp estimates for composition in Sobolev and Zygmund spaces. In order to prove that the new class of operations benefits of symbolic calculus properties when composed by a paradifferential operator, we discuss the pull-back of pseudodifferential and paradifferential operators which then become Fourier Integral Operators. In this discussion we show that those Fourier Integral Operators obtained by pull-back are pseudodifferential or paradifferential operators if and only if they are pulled-back by a diffeomorphism that is a change of variable. We give a proof of the change of variables in paradifferential operators.

Finally we study the cutoff defining paradifferential operators and it's stability by successive composition. It is known that the cutoff becomes worse after each composition, we give a slightly refined version of the cutoffs proposed by H\"ormander in  \cite{Hormander97} for which give an optimal estimate on the support of the cutoff after composition.\\ 
\keyword{Composition, Paracomposition, Paradifferential operators, Change of variables.} 
\end{abstract}

\maketitle


\vspace{-5mm}

\tableofcontents

\vspace{-7mm}

\section{Introduction}
One of the goals of this paper is to build upon the following construction by Alinhac: given a $\rho>0$ and $C^{1+\rho}$ diffeomorphism $\chi:\Omega_1\rightarrow \Omega_2$ between two open subsets of $\r^d$, Alinhac constructed an operator $\chi^*:\mathfrak{D} '(\Omega_2)\rightarrow \mathfrak{D}'(\Omega_1)$
having analogous properties to the usual composition $u \rightarrow u \circ \chi$ but with limited dependency on the regularity of $\chi$ as for classical paradifferential operators that is the paraproduct $T_a$ is well defined from $H^s\rightarrow H^s$, for all $s$ for $a$ merely in  $ L^\infty$.

Alinhac's construction was motivated by questions that arose from the study of non linear PDEs for example: the study of the transport of a distribution's wave front by a diffeomorphism with low regularity as in the works of E. Leichtnam in \cite{Leichtnam84}, the study of the singularities of solutions to semi-linear hyperbolic evolution problems and the characteristic surfaces of the associated operators(here having low regularity), the main reference being Bony's work on the subject (\cite{Bony81},\cite{Bony79},\cite{Bony82},\cite{Bony83}). More recently in \cite{Alazard11} and \cite{Alazard14}, the Paracomposition appears naturally as the ``good variable"\footnote{The so called good unknown of Alinhac.} after a low regularity change of variable in treating the Cauchy problem for the Water Waves system with rough data. It also appears in our recent proof of the quasi-linearity of the Water Waves system \cite{Ayman19}.

Finally the construction of $\chi^*$ gives a complete linearisation formula to the composition of two functions(with one being a diffeomorphism) generalizing the classic para-linearisation theorem by Bony \cite{Bony81} in a low regularity case. 

Bony showed that for $u\in C^ \infty$ and $\chi \in H^s_{loc}, s>\frac{d}{2}$ (without the diffeomorphism hypothesis):
\[
u\circ \chi=T_{u'(\chi)}\chi +\ \text{remainder},
\]
and Alinhac showed for $u\in C^\sigma_{loc}, \sigma>1$ and $\chi \in C^{1+\rho}, \rho>0$ a diffeomorphism:
\begin{equation}\label{paracomposition_introduction_paralinearisation of compostion}
u\circ \chi=\chi^*u+T_{u'(\chi)}\chi +\ \text{remainder}.
\end{equation}

Another fundamental result obtained by Alinhac is that the operator $\chi^*$ benefits from symbolic calculus properties, that is, it conjugates paradifferential operators. Given $T_h$ a paradifferential operator, Alinhac proved a result in the form:
\[\chi^*T_hu=T_{h^*}\chi^*u+\ \text{remainder},\]
 where $h^*$ is the pulled back symbol in the case of diffeomorphisms.\\
  
The main result of this work on paracomposition generalizes Bony's and Alinhac's work by:
\begin{itemize}
 \item dropping the diffeomorphism hypothesis with a new operator $\chi^\star:\mathfrak{D} '(\Omega_2)\rightarrow \mathfrak{D}'(\Omega_1)$. $\chi^\star$ will coincide with Alinhac's operator $\chi^*$ modulo a regular remainder in the case of diffeomorphisms.
 \item Giving estimates in ``global" spaces which were of interest for us in our study of the flow map regularity associated to the Water Waves system.
 \end{itemize}
 We will then show that $\chi^\star$ benefits of symbolic calculus properties, for that we will start by discussing the pull-back of pseudodifferential and paradifferential operators by $\chi$ which then become Fourier integral operators. In this discussion we show that those Fourier Integral Operators obtained by pull-back are pseudodifferential or paradifferential operators if and only if they are pulled-back by a diffeomorphism that is a change of variable. We also give a proof to the change of variables in paradifferential operators as we could not find a reference in the literature.\\ 
 
Taking advantage of the structure of the paper, where we recall various definition in microlocal analysis, we discuss the regularisation by cutoff in the definition of paradifferential operators and it's relation to composition of two operators. Taking a symbol $a\in \Gamma^m_0$ with limited regularity we define the regularisation by cutoff:
\[
\fr_x\sigma_a(\xi,\eta)=\psi(\xi,\eta) \fr_x a(\xi,\eta),
\]
where $\phi$ is a cut-off function with support bounded away from $(\eta,0)$ and $(-\eta,\eta)$ at infinity.

We look at two types of regularisation through cutoffs, $(\psi_H^B)_{B>2}$ defined by H\"ormander in \cite{Hormander97}:
\begin{equation}\label{paracomposition_introduction_Hormander cutoff}
\psi_H^B(\eta,\xi)=0 \text{ when }
\begin{cases}
\abs{\eta}>B(\abs{\xi}+1),\\
\abs{\xi}>B(\abs{\eta+\xi}+1),
\end{cases}
\hspace*{-0,4cm}
\text{and }
\psi_H^B(\eta,\xi)=1 \text{ when } \abs{\xi}>B(\abs{\eta}+1),
\end{equation}
 and $(\psi^\eps_M)_{\eps<1}$ defined by M\'etivier in \cite{Metivier08}:
\begin{equation}\label{paracomposition_introduction_Metivier cutoff}
\psi_H^B(\eta,\xi)=0 \text{ when }
\abs{\eta}\geq \eps (\abs{\xi}+1),
\text{ and }
\psi_H^B(\eta,\xi)=1 \text{ when } \abs{\eta}\leq \frac{\eps}{2} (\abs{\xi}+1).
\end{equation}
The following figures illustrate the choice of cutoff functions in the plane $(\xi,\eta)$ when $d=1$:
\begin{figure}[h!]
\begin{tikzpicture}[xscale=1.3,yscale=0.7,dot/.style={circle,inner sep=1pt,fill,label={#1},name=#1},
 extended line/.style={shorten >=-#1,shorten <=-#1},
 extended line/.default=1cm]
\draw [->] (0,-1.5) -- (0,7.7);
\draw [->] (-4,3) -- (4,3);
\node[right] at (0,7.6) {\scriptsize $\xi$};
\node[below right] at (4,3.2) {\scriptsize $\eta$};
\coordinate (Q1) at (0,5);
\node [fill=black,inner sep=1pt,label=0:$B$] at (Q1) {};
\coordinate (Q2) at (1,7.5);
\coordinate (Q3) at (-1,7.5);
\draw[thick,black,pattern=north east lines,pattern color=black] (Q3) -- (Q1)--(Q2);
\node[above] at (Q3) {\scriptsize $\xi = -B\eta+B$};
\node[above ] at (Q2) {\scriptsize $\xi = B\eta+B$};
\node at (0.8,6.2) {\scriptsize $\psi = 1$};
\coordinate (SQ1) at (0,1);
\node [fill=black,inner sep=1pt,label=0:$-B$] at (SQ1) {};
\coordinate (SQ2) at (1,-1.5);
\coordinate (SQ3) at (-1,-1.5);
\draw[thick,black,pattern=north east lines,pattern color=black] (SQ3) -- (SQ1)--(SQ2);
\node at (0.8,0.2) {\scriptsize $\psi = 1$};
\node[below] at (SQ3) {\scriptsize $\xi = B\eta-B$};
\node[below] at (SQ2) {\scriptsize $\xi = -B\eta-B$};
\coordinate (0) at (0,3);
\coordinate (SD1) at (-4,7);
\node[left] at (SD1) {\scriptsize $\xi = -\eta$};
\coordinate (SD2) at (4,-1);
\draw (SD1) -- (SD2);
\coordinate (-B) at (-2.5,3);
\node [fill=black,inner sep=1pt,label=-45:$-B$] at (-B) {};
\coordinate (-B') at (-2.5,5.5);
\draw[dotted] (-B) -- (-B');
\coordinate (P1) at (-4,8);
\coordinate (P2) at (-4,6);
\draw[thick,black,pattern=north west lines,pattern color=red] (P1) -- (-B')--(P2);
\node[left] at (P1) {\scriptsize $\xi = \frac{B(\eta+1)}{1-B}$};
\node[left] at (P2) {\scriptsize $\xi = \frac{B(1-\eta)}{1+B}$};
\node[above right] at (-B') {\scriptsize $\psi = 0$};
\coordinate (B) at (2.5,3);
\node [fill=black,inner sep=1pt,label=-45:$B$] at (B) {};
\coordinate (B') at (2.5,0.5);
\draw[dotted] (B) -- (B');
\coordinate (P3) at (4,0);
\coordinate (P4) at (4,-2);
\draw[thick,black,pattern=north west lines,pattern color=red] (P3) -- (B')--(P4);\node[right] at (P4) {\scriptsize $\xi = \frac{B(\eta-1)}{1-B}$};
\node[right] at (P3) {\scriptsize $\xi = \frac{-B(1+\eta)}{1+B}$};
\node[above right] at (B') {\scriptsize $\psi = 0$};
\coordinate (R2) at (4,3.6);
\coordinate (R3) at (4,2.4);
\draw[thick,black,pattern=north west lines,pattern color=red] (R3) -- (B)--(R2);
\node[above] at (B) {\scriptsize $\psi = 0$};
\node[right] at (R2) {\scriptsize $\xi = \frac{\eta}{B}-1$};
\node[right] at (R3) {\scriptsize $\xi = -\frac{\eta}{B}+1$};
\coordinate (SR2) at (-4,3.6);
\coordinate (SR3) at (-4,2.4);
\draw[thick,black,pattern=north west lines,pattern color=red] (SR3) -- (-B)--(SR2);
\node[above] at (-B) {\scriptsize $\psi = 0$};
\node[left] at (SR2) {\scriptsize $\xi = -\frac{\eta}{B}-1$};
\node[left] at (SR3) {\scriptsize $\xi = \frac{\eta}{B}+1$};
\end{tikzpicture}
\caption{H\"ormander's choice of cut-off function $(\psi_H^B)_{B>2}$.}
\label{paracomposition_introduction_figure1}
 \end{figure}
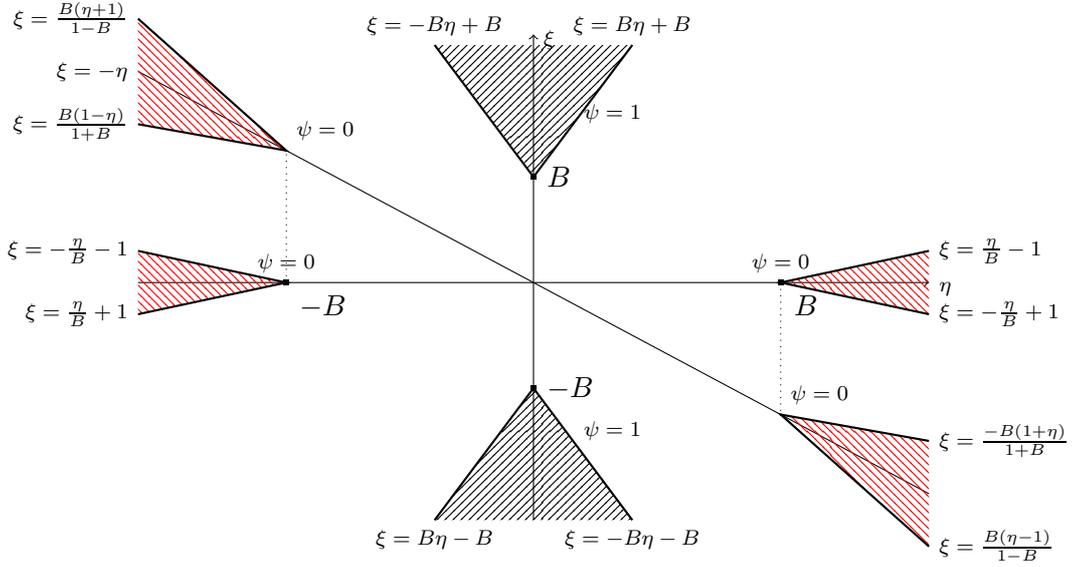

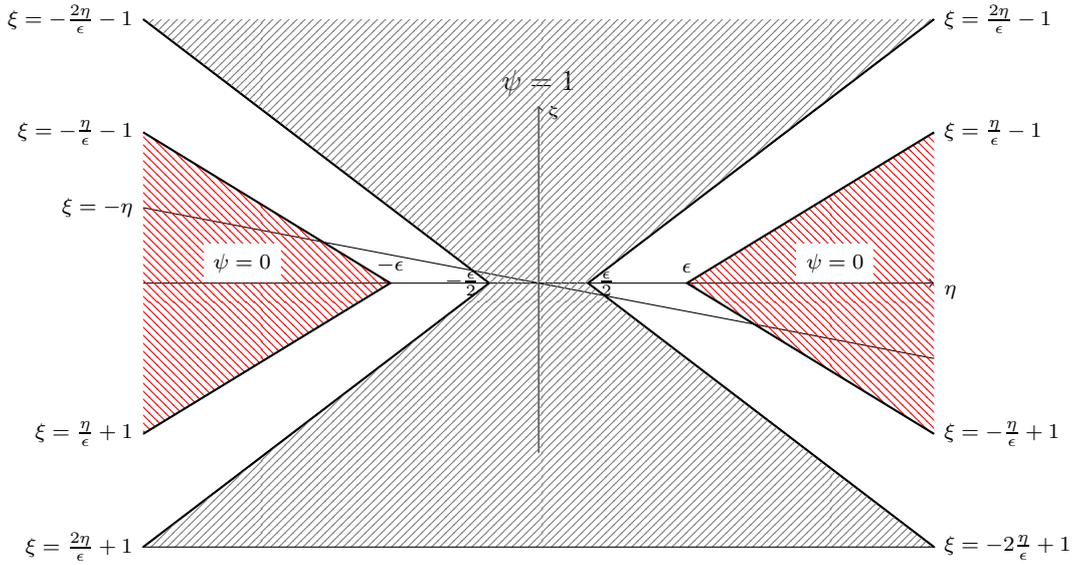
\begin{figure}[h!]
\begin{tikzpicture}[xscale=1.3,yscale=0.5,dot/.style={circle,inner sep=1pt,fill,label={#1},name=#1},
 extended line/.style={shorten >=-#1,shorten <=-#1},
 extended line/.default=1cm]
\draw [->] (0,-1.5) -- (0,7.7);
\draw [->] (-4,3) -- (4,3);
\node[right] at (0,7.6) {\scriptsize $\xi$};
\node[below right] at (4,3.2) {\scriptsize $\eta$};
\coordinate (0) at (0,3);
\coordinate (SD1) at (-4,5);
\node[left] at (SD1) {\scriptsize $\xi = -\eta$};
\coordinate (SD2) at (4,1);
\draw (SD1) -- (SD2);
\node[above,fill=white] at (0,7.7) {$\psi = 1$};
\coordinate (-B1) at (-1.5,3);
\coordinate (-B2) at (-0.5,3);
\coordinate (B1) at (1.5,3);
\coordinate (B2) at (0.5,3);
\coordinate (R2) at (4,7);
\coordinate (R2') at (4,10);
\coordinate (R3) at (4,-1);
\coordinate (R3') at (4,-4);
\draw[thick,black,pattern=north west lines,pattern color=red] (R3) -- (B1)--(R2);
\draw[thick,black] (R3') -- (B2)--(R2');
\node[above] at (B1) {\scriptsize $\eps$};
\node[right] at (B2) {\scriptsize \textbf{$\frac{\eps}{2}$}};
\node[above] at (-B1) {\scriptsize $-\eps$};
\node[left] at (-B2) {\scriptsize \textbf{$-\frac{\eps}{2}$}};
\node[above,fill=white] at (3,3) {\scriptsize $\psi = 0$};
\node[right] at (R2) {\scriptsize $\xi = \frac{\eta}{\eps}-1$};
\node[right] at (R2') {\scriptsize $\xi = \frac{2\eta}{\eps}-1$};
\node[right] at (R3) {\scriptsize $\xi = -\frac{\eta}{\eps}+1$};
\node[right] at (R3') {\scriptsize $\xi = -2\frac{\eta}{\eps}+1$};
\coordinate (SR2) at (-4,7);
\coordinate (SR3) at (-4,-1);
\coordinate (SR2') at (-4,10);
\coordinate (SR3') at (-4,-4);
\draw[thick,black,pattern=north west lines,pattern color=red] (SR3) -- (-B1)--(SR2);
\draw[thick,black] (SR3') -- (-B2)--(SR2');
\node[above,fill=white] at (-3,3) {\scriptsize $\psi = 0$};
\node[left] at (SR2) {\scriptsize $\xi = -\frac{\eta}{\eps}-1$};
\node[left] at (SR3) {\scriptsize $\xi = \frac{\eta}{\eps}+1$};
\node[left] at (SR2') {\scriptsize $\xi = -\frac{2\eta}{\eps}-1$};
\node[left] at (SR3') {\scriptsize $\xi = \frac{2\eta}{\eps}+1$};
\draw[pattern=north east lines,pattern color=gray] (R2') -- (B2)--(R3')--(SR3') -- (-B2)--(SR2');
\end{tikzpicture}
\caption{M\'etivier choice of cut-off function $(\psi^\eps_M)_{\eps<1}$.}
\label{paracomposition_introduction_figure2}
 \end{figure}
The effect of the composition on the support of cutoff is seen by the following:
\begin{equation}\label{paracomposition_effect comp Hormander Metivier}
\sigma^{\psi^\eps_M}_a \circ \sigma^{\psi^\eps_M}_a=\sigma^{\psi^{2\eps+\eps^2}_M}_{a\otimes a} \text{ and } \sigma^{\psi_H^B}_a \circ \sigma^{\psi_H^B}_a=\sigma^{\psi_H^{\frac{B^2}{2B+1}}}_{a\otimes a}.
\end{equation}
Thus the composition of two paradifferential operators with cutoffs $(\psi_H^B)_{B>2}$ and $(\psi^\eps_M)_{\eps<1}$ are still paradifferential operators but with worse cutoffs $(\psi_H^{\frac{B^2}{2B+1}})_{B>2}$ and $\psi^{2\eps+\eps^2}_M$. This is not a problem when considering a finite number of composition of paradifferential operators but it becomes crucial if one for example needs to understand the limit of the series, for example $\sum \frac{(T_1)^k}{k!}$. The $k-$th order term of such a sum has the cutoffs with parameters $\sim k\eps$
and $\sim \frac{B}{2^k}$ which are no longer paradifferential operators when the conditions $\frac{B}{2^k}>2$ and $k\eps>1$ are no longer verified. 

Thus the class of para-differential operators is not stable by composition. On the other hand it is a subclass of operators that belong to the closed algebra (as shown by Bourdaud in \cite{Bourdaud88}) $\bold \Psi^m_{1,1}=S^m_{1,1}\cap {}^t\big(S^m_{1,1}\big)$, where we use the notation ${}^t A$ for the adjoint of an operator $A$ in order to avoid confusion with the pull-back operator. \\
To study the lack of stability of the class of para-differential operators in $\bold \Psi^m_{1,1}$ we look more closely to the cutoffs $\psi_H^B/\psi^\eps_M$. We first note that the estimate \eqref{paracomposition_effect comp Hormander Metivier} is optimal in the respective sub-classes of cut-offs as can be seen by studying the para-product operator $\sigma^{\psi_H^B/\psi^\eps_M}_{e^{ix\xi_0}}$, with $\xi_0\in \z$ fixed and $x\in \t$, on the torus. Here we seek a refined version of those estimates for this we introduce a special class of cutoffs $(\psi^{B_1,B_2,b})_{B_1>0,B_2>1,b>0}$, which is included in a modified version of the H\"ormader and M\'etrivier classes of cutoffs. 
\begin{definition}\label{paracomposition_introduction_new cutoff}
For $B_1>0, B_2>1$, $b>0$ and $i\in \set{1,\cdots,d}$, $\psi^{B_1,B_2,b}$ is defined by
\[
\begin{cases}\psi^{B_1,B_2,b}(\eta,\xi)=0 \text{ when }
\xi_i< B_1\eta_i+b,
\\
\psi^{B_1,B_2,b}(\eta,\xi)=1 \text{ when } \xi_i>B_1\eta_i+b+1,
\end{cases}\text{ for } \xi_i\geq,\eta_i\geq 0,
\]
\[
\begin{cases}\psi^{B_1,B_2,b}(\eta,\xi)=0 \text{ when }
\xi_i< -B_2\eta_i+b,
\\
\psi^{B_1,B_2,b}(\eta,\xi)=1 \text{ when } \xi_i>-B_2\eta_i+b+1,
\end{cases}\text{ for } \xi_i\geq 0,\eta_i\leq 0,
\]
and for $\xi_i\leq 0$, $\psi^{B_1,B_2,b}(\cdots,\xi_i,\cdots,\eta_i,\cdots)=\psi^{B_1,B_2,b}(\cdots,-\eta_i,\cdots,\xi_i,\cdots)$.
\end{definition}
In the plane $(\eta,\xi)$, when $d=1$, this is illustrated by the following figure:
\begin{figure}[h!]
\begin{tikzpicture}[xscale=2,yscale=0.6,dot/.style={circle,inner sep=1pt,fill,label={#1},name=#1},
 extended line/.style={shorten >=-#1,shorten <=-#1},
 extended line/.default=1cm]
\draw [->] (0,-1.5) -- (0,8.7);
\draw [->] (-2,3) -- (2,3);
\node[right] at (0,8.6) {\scriptsize $\xi$};
\node[below right] at (2,3.2) {\scriptsize $\eta$};
\coordinate (0) at (0,3);
\coordinate (SD1) at (-2,5);
\node[left] at (SD1) {\scriptsize $\xi = -\eta$};
\coordinate (SD2) at (2,1);
\draw (SD1) -- (SD2);
\coordinate (Q1) at (0,5);
\node [fill=black,inner sep=1pt,label=0:\tiny $b+1$] at (Q1) {};
\coordinate (Q2) at (2,8);
\coordinate (Q3) at (-1.5,8.75);
\coordinate (Q1') at (0,4);
\node [fill=black,inner sep=1pt,label=0:\tiny $b$] at (Q1') {};
\coordinate (Q2') at (2,7);
\coordinate (Q3') at (-2,9);
\draw[thick,black,pattern=north east lines,pattern color=black] (Q3) -- (Q1)--(Q2);
\node[above right] at (Q3) {\scriptsize $\xi = -B_2\eta+b+1$};
\node[above left] at (Q2) {\scriptsize $\xi = B_1\eta+b+1$};
\node[above left] at (Q3') {\scriptsize $\xi = -B_2\eta+b$};
\node[above right] at (Q2') {\scriptsize $\xi = B_1\eta+b$};
\node[fill=white] at (0,6.5) {\scriptsize $\psi = 1$};
\node[fill=white] at (1,3.8) { $\psi = 0$};
\coordinate (SQ1) at (0,1);
\node [fill=black,inner sep=1pt,label=0:\tiny $-b-1$] at (SQ1) {};
\coordinate (SQ2) at (1.5,-2.75);
\coordinate (SQ3) at (-2,-2);
\coordinate (SQ1') at (0,2);
\node [fill=black,inner sep=1pt,label=0:\tiny $-b$] at (SQ1') {};
\coordinate (SQ2') at (2,-2);
\coordinate (SQ3') at (-2,0);
\draw[thick,black,pattern=north east lines,pattern color=red] (Q3') -- (Q1')--(SQ1')--(SQ3');
\draw[thick,black,pattern=north east lines,pattern color=red] (Q2') -- (Q1')--(SQ1')--(SQ2');
\draw[thick,black,pattern=north east lines,pattern color=black] (SQ3) -- (SQ1)--(SQ2);
\node[fill=white] at (0,-0.2) {\scriptsize $\psi = 1$};
\node[below] at (SQ3) {\scriptsize $\xi = B_1\eta-b-1$};
\node[below] at (SQ2) {\scriptsize $\xi = -B_2\eta-b-1$};
\node[below] at (SQ3') {\scriptsize $\xi = B_1\eta-b$};
\node[below] at (SQ2') {\scriptsize $\xi = -B_2\eta-b$};

\end{tikzpicture}
\caption{The choice of cut-off function $(\psi^{B_1,B_2,b})_{B_1>0,B_2>1,b>0}$, $d=1$.}
\label{paracomposition_introduction_figure3}
 \end{figure}
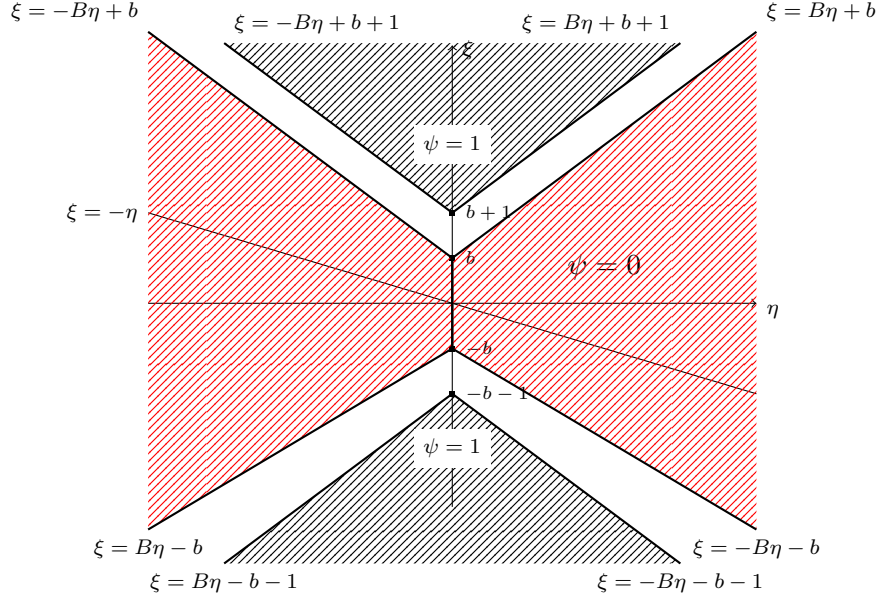

From this new definition of asymmetrical type cut-offs and H\"ormander characterisation of $\bold \Psi^m_{1,1}$ we will deduce the following:
\begin{theorem}
Consider three real numbers $\rho\geq 0$, $B_1>0$, $B_2>1$, $b>0$, and two symbols $a \in\Gamma_\rho^\alpha$ and $b \in\Gamma_\rho^\beta$. When taking adjoints we get that there exists $a^t \in \Gamma_\rho^{\alpha}$ such that 
$$\left(T^{\psi^{B_1,B_2,b}}_a\right)^t=T^{\psi^{B_2-1,B_1+1,b}}_{a^t}.$$

For composition  there exists $a\otimes b \in \Gamma_\rho^{\alpha+\beta}$ such that for $B_1>1$
$$T^{\psi^{B_1,B_2,b}}_a\circ T^{\psi^{B_1,B_2,b}}_b=T^{\psi^{\frac{B_1^2}{2B_1-1},\frac{B_2^2}{2B_2+1}},b}_{a\otimes b}.$$
\end{theorem}

The previous theorem shows a key asymmetrical phenomena happening when composing and taking adjoints of paradifferential operators that could not be captured throuh the standard symmetrical cut-off operators defined previously which motivated the definition of the new asymmetrical type cut-offs to capture it.

\subsection{An application of the paracomposition operator}

A main application of the paracomposition operator and the paralinearisation formula \eqref{paracomposition_introduction_paralinearisation of compostion} now that we have the global estimates is the study of composition in Sobolev spaces with limited regularity. The literature on this problem is rich and our knowledge of it is certainly incomplete but we mainly looked on two recent articles treating this subject \cite{Bauer19} and \cite{Inci18} in which they study composition in Sobolev spaces and the geometry of diffeomorphisms groups on manifolds. We will limit the discussion here to the Euclidean space in which the tools presented here significantly improve upon the results from \cite{Bauer19} and \cite{Inci18}. First in \cite{Inci18} the composition estimates are proven on $H^n(\r^d)\times D^s(\r^d)$ with $n\in \n$, $s>1+\frac{d}{2}$ an integer and 
$$D^s(\r^d)=\set{\psi -id\in H^{s}(\r^d), \psi \text{ is a diffeomorphism}}.$$
Here we generalize this to $n,s$ real number and from the paralinearisation formula \eqref{paracomposition_introduction_paralinearisation of compostion} it is justified to work in the class $D^s(\r^d)$ which appears naturally but it admits several generalization the simplest one is for example using Zygmund spaces. We also clarify the need of the diffeomorphism hypothesis. More precisely we have the following,
\begin{corollary}
Consider two real numbers $s\in \r$, $\rho \in R_+^* \setminus \n$, and take $\phi \in H^s(\r^d)$ and consider $\chi\in W^{1+\rho,\infty}_{loc}(\r^d)$ a diffeomorphism such that $D\chi\in W^{\rho,\infty}(\r^d)$. Then $\phi \circ \chi \in H^{min(s,\rho)}(\r^d)$.
\end{corollary}
The result we have is even stronger indeed it's a Kato-Ponce like decomposition of the different terms that appear in the $H^s$ estimates of composition, for example keeping the notations of the previous Corollary and taking $\psi \in D^s(\r^d)$ we can have estimates of the form:
\[
\norm{\phi \circ \psi}_{H^s}\leq \norm{D\psi}_{L^\infty}\norm{\psi}_{H^s}+\norm{D\phi}_{L^\infty}\norm{\psi-Id}_{H^s}.
\]
So if we were only working with Sobolev spaces more sophisticated versions of the previous inequality give, 
\begin{corollary}
Consider a real number $s>1+\frac{d}{2}$, and take $\phi \in H^s(\r^d)$ and consider $\chi\in W^{1+s-\frac{d}{2},\infty}_{loc}(\r^d)$ a diffeomorphism such that $D\chi \in W^{1,\infty}(\r^d)$ and $D^2\chi\in H^{s-2}(\r^d)$. Then $\phi \circ \chi \in H^{s}(\r^d)$.
\end{corollary}
Secondly in \cite{Bauer19} to prove the well posedness of EPDIFF equation they treat the case of change of variables in pseudodifferential operator with a diffeomorphism with limited regularity. The results are restricted to skew-symmetric operators with compact support and a diffeomorphism in the class $D^s(\r^d)$. Here with the paradifferential calculus and the paracomposition in hand, the more general case of symbols with limited regularity is treated, the pseudodifferential symbols being the the case where the symbols are regular, the ellipticity and symmetry hypothesis dropped and the need of diffeomorphisms justified. More precisely we have
\begin{corollary}
Consider a real number $r$, $A \in S^r(\r^d\times \r^d)$ and $\chi\in W^{1+s-\frac{d}{2},\infty}_{loc}(\r^d)$ a diffeomorphism such that $D\chi \in W^{1,\infty}(\r^d)$ and $D^2\chi\in H^{s-2}(\r^d)$. Then the pull back $A^*$ of $A$ by $\chi$ defined as 
\[u\in \sr, A^*u=[A\big(u\circ \chi \big)]\circ \chi^{-1},\]
is extended to a linear bounded operator from $H^s(\r^d)$ to $H^{s-r}(\r^d)$.
\end{corollary}

\subsection{Heuristics behind Paradifferential calculus and Paracomposition} 
 For the sake of this discussion let us pretend that $\partial_x$ is left-invertible with a choice of $\partial_x^{-1}$ that acts continuously from $H^s$ to $H^{s+1}$. We follow here analogous ideas to the ones presented by Shnirelman in \cite{Shnirelman05}.
 
\subsubsection*{\textbf{Paraproduct}} 
One way to define the paraproduct of two functions $f,g\in H^s$ with $s$ sufficiently large is: we differentiate $fg$ $k$ times, using the Leibniz formula, and then restore the function $fg$ by the $k$-th power of $\partial_x^{-1}$:
 \begin{align*}
 fg&=\partial_x^{-k}\partial_x^{k}(fg)\\
 	&=\partial_x^{-k}\big(g\partial_x^k f+k\partial_x g\partial_x^{k-1} f+\dots+k\partial_x f\partial_x^{k-1} g+g\partial_x^k f  \big)\\
 	&=T_g f+T_fg+R,
 \end{align*}
 where,
 \[T_gf=\partial_x^{-k}\big(g\partial_x^k f\big), \ \ T_fg=\partial_x^{-k}\big(f\partial_x^k g\big),\]
 and $R$ is the sum of all remaining terms.
 
 The key observation is that if $s>\frac{1}{2}+k$, then $g \mapsto T_fg$ is a continuous operator in $H^s$ for $f  \in H^{s-k}$. The remainder $R$ is a continuous bilinear operator from $H^s$ to $H^{s+1}$. 
 
 The operator $T_fg$ is called the paraproduct of $g$ and $f$ and can be interpreted as follows. The term $T_fg$ takes into play high frequencies of $g$ compared to those of $f$ and demands more regularity in $g\in H^s$ than $f \in H^{s-k}$ thus the term $T_fg$ bears the ``singularities" brought on by $g$ in the product $fg$. Symmetrically $T_gf$ bears the ``singularities" brought on by $f$ in the product $fg$ and the remainder $R$ is a smoother function ($H^{s+1}$) and does not contribute to the main singularities of the product.
 
 Notice that this definition uses a ``general" heuristic from PDE that is the worst terms are the highest order terms (ones involving the highest order of differentiation).
 
\subsubsection*{\textbf{Paracomposition}}
We again work with $f \in H^s$ and $g \in C^s$ with $s$ large and consider the composition of two functions $f\circ g$ which bears the singularities of both $f$ and $g$, and our goal is to separate them. We proceed as before by differentiating $f \circ g$ $k$ times, using the Fa\'a di Bruno's formula, and then restore the function $fg$ by the $k$-th power of $\partial_x^{-1}$:
\begin{align*}
 f \circ g&=\partial_x^{-k} \partial_x^{k} (f \circ g)\\
 	&=\partial_x^{-k}\big((\partial_x^k f\circ g)\cdot(\partial_x g)^k 
 	+\dots+(\partial_x f\circ g)\cdot\partial_x^k g  \big)\\
 	&=g^*f+T_{\partial_x f \circ g}g+R,
 \end{align*}
 where,
 \[g^*f=\partial_x^{-k}\big((\partial_x^k f\circ g)\cdot(\partial_x g)^k \big) \text{ is the paracomposition of $f$ by $g$}\]
 and $R$ is the sum of all remaining terms.
 
 Again the key observation is that if $s>\frac{1}{2}+k$, then $f \mapsto g^*f$ is a continuous operator in $H^s$ for $g  \in C^{s-k}$. Thus this term bears essentially the singularities of $f$ in $f\circ g$. As before $T_{\partial_x f \circ g}g$ bears essentially the singularities of $g$ in $f\circ g$. The remainder $R$ is a continuous bilinear operator from $H^s$ to $H^{s+1}$. Thus we have separated the singularities of the composition $f\circ g$.
\subsubsection*{\textbf{Change of variable in Paradifferential operators}}
 From what we have seen previously it seems likely that the adequate change of variable for paradifferential operators is one that comes from commuting with the paracomposition by a diffeomorphism. We carry on the previous computation with the trivial paradifferential operator $\partial_x \sim T_{i\xi}$ and we suppose moreover that $g$ is a diffeomorphism.
\begin{align*}
 g^* \partial_x f &= \partial_x^{-k}\big((\partial_x^{k+1} f\circ g)\cdot(\partial_x g)^k \big)\\
 &= \partial_x^{-k}\big(\partial_x^{k}[\partial_x^{-k}(\partial_x^{k+1} f\circ g)\cdot(\partial_x g)^{k+1}]\cdot(\partial_x g)^{-1} \big)\\
 &=T_{(\partial_x g)^{-1}}T_{i \xi} g^*f,
 \end{align*} 
 and we notice that $(\partial_x g)^{-1}i \xi=(\partial_x)^*$ is the usual pull-back formula for pseudodifferential symbols by a diffeomorphism $g$, giving us the desired symbolic calculus rules.
\subsection{Structure of the paper}
Given the technical nature of the results in this paper we start the paper by a quick overview in sections \ref{paracomposition_section Notations and functional analysis} and \ref{paracomposition_section Notions of microlocal analysis}  of notions of functional analysis and microlocal analysis. At the end of section \ref{paracomposition_section Notions of microlocal analysis}  we discuss and show the different properties associated to the different cutoffs of paradifferential operators presented above. Then in section \ref{paracomposition_section Pull-back of pseudo and para- differential operators} we present the different results on the change of variables in pseudodifferential and paradifferential operators. And finally with all of the tools needed we redefine the paracomposition in section \ref{paracomposition_section Paracomposition} and show that it satisfies all of the desired properties. Thus the reader interested in series of paradifferential operators can go directly to section \ref{paracomposition_section Notions of microlocal analysis} , if she/he is interested only in the change of variable/pull-back can go to section \ref{paracomposition_section Pull-back of pseudo and para- differential operators} and if she/he is interested only in the paracomposition she/he can go to section \ref{paracomposition_section Paracomposition}.

\subsection{Acknowledgement}
I would like to express my sincere gratitude to my thesis advisor Thomas Alazard.

\section{Notations and functional analysis}\label{paracomposition_section Notations and functional analysis}
We present the definitions of the functional spaces that will be used.\\
We will use the usual definitions and standard notations for the regular functions $C^k$, $C^k_0$ for those with compact support, the distribution space $\dr'$,$\er'$ for those with compact support, $\dr'^k$,$\er'^k$ for distributions of order k, Lebesgue spaces ($L^p$), Sobolev spaces ($H^s,W^{p,q}$) and the Schwartz class $\sr$ and it's dual $\sr'$. All of those spaces are equipped with their standard topologies. We also use the \textit{Landau notation}  $O_{\norm{ \ }}(X)$.

\begin{notation}
We will use $\d$ to denote $\t$ or $\r$ and $\hat{\d}$ to denote their duals that is $\z$ in the case of $\t$ and $\r$ in the case of $\r$. For concision an integral on on $\z^d$ that is $\displaystyle \int_{\z^d}$ should be understood as $\displaystyle \sum_\z^d$. A function $a$ is said to be in $ C^\infty(\t^d \times \z^d)$ if for every $\xi \in \z$ $a(\cdot,\xi) \in C^\infty(\t^d)$. For $\xi\in \z^d$ and $i \in \set{1,\cdots,d}$, $\partial_{\xi_i}$ should be understood as the partial forward difference operator, that is
\[\partial_{\xi_i} a(\xi_1,\cdots,\xi_i,\cdots,\xi_d)=a(\xi_1,\cdots,\xi_i+1,\cdots,\xi_d)-a(\xi_1,\cdots,\xi_i,\cdots,\xi_d),\ \xi \in \z^d.\]
We recall the following simple identities for the Fourier transform on the Torus:
\[
\begin{cases}
\fr_{\t^d}(\partial_x^\alpha f)(\xi)=\xi^\alpha \fr_{\t^d}(f)(\xi), \xi \in \z^d,\\
\fr_{\t^d}((e^{-2i\pi x}-1)^\alpha f)(\xi)=\xi^\alpha \fr_{\t^d}(f)(\xi), \xi \in \z^d, \ x\in \t^d.
\end{cases}
\]
\end{notation}

\begin{definition}[Littlewood-Paley decomposition]\label{paracomposition_section Notations and functional analysis_def LP Theory}
Pick $P_0\in C^\infty_0(\r^d)$ so that, $$P_0(\xi)=1 \text{ for }\abs{\xi}<1 \text{ and }0\text{ for } \abs{\xi}>2 .$$ We define a dyadic decomposition of unity by:
\[ \text{for } k \geq 1, \ P_{\leq k}(\xi)=P_0(2^{-k}\xi), \ P_k(\xi)=P_{\leq k}(\xi)-P_{\leq k-1}(\xi). \]
 Thus,\[ P_{\leq k}(\xi)=\sum_{0\leq j \leq k}P_j(\xi) \text{ and } 1=\sum_{j=0}^\infty P_j(\xi). \]
 Introduce the operator acting on $\mathscr S '(\r^d)$: 
 \[P_{\leq k}u=\fr^{-1}(P_{\leq k}(\xi)u) \text{ and } u_k=\fr^{-1}(P_k(\xi)u).\]
 Thus,
 \[u=\sum_k u_k.\]
 Finally put $\set{k\geq 1, C_k=\supp \ P_k}$ the set of rings associated to this decomposition.
\end{definition}

An interesting property of the Littlewood-Paley decomposition is that even if the decomposed function is merely a distribution the terms of the decomposition are regular, indeed they all have compact spectrum and thus are entire functions. On classical functions spaces this regularisation effect can be ``measured" by the following inequalities due to Bernstein.

\begin{proposition}[Bernstein's inequalities]\label{paracomposition_Notations and functional analysis_bernstein1}
Suppose that $a\in L^p(\r^d)$ has its spectrum contained in the ball $\set{\abs{\xi}\leq \lambda}$. Then $a\in C^\infty$ and for all $\alpha \in  \n^d$ and $1\leq p \leq q \leq +\infty$, there is $C_{\alpha,p,q}$ (independent of $\lambda$) such that 
\[\norm{\partial^{\alpha}_x a}_{L^q} \leq C_{\alpha,p,q} \lambda^{\abs{\alpha}+\frac{d}{p}-\frac{d}{q}}\norm{a}_{L^p}.\]
In particular,
\[\norm{\partial^{\alpha}_x a}_{L^q} \leq C_{\alpha} \lambda^{\abs{\alpha}}\norm{a}_{L^p}, \text{ and for $p=2$, $p=\infty$}\]
\[\norm{a}_{L^\infty}\leq C \lambda^{\frac{d}{2}} \norm{a}_{L^2}.\]
If moreover a has it's spectrum is in $ \set{0<\mu \leq \abs{\xi}\leq \lambda}$ then:
\[
 C_{\alpha,q}^{-1} \mu^{\abs{\alpha}}\norm{a}_{L^q}\leq \norm{\partial^{\alpha}_x a}_{L^q} \leq C_{\alpha,q} \lambda^{\abs{\alpha}}\norm{a}_{L^q}.
\]
\end{proposition}

\begin{proposition}\label{paracomposition_Notations and functional analysis_bernstein2}
For all $\mu >0$, there is a constant $C$ such that for all $\lambda>0$ and for all $\alpha \in W^{\mu,\infty}$ with spectrum contained in $\set{\abs{\xi}\geq \lambda}$. one has the following estimate: 
\[\norm{a}_{L^\infty}\leq C \lambda^{-\mu} \norm{a}_{W^{\mu,\infty}}.\]
\end{proposition}

\begin{definition}[Singular support]
 $f \in \sr'(\r^d)$ is said to be $C^\infty$ in a neighborhood of x, if there exists a neighborhood $\omega$ of x such that for all $\psi \in C^\infty_0(\omega)$ we have $\psi f \in C^\infty(\r^d)$.\\
The singular support of a distribution f, $\singsupp f$, is defined as the complementary of such points and is clearly closed. 
\end{definition}

\begin{definition}[Zygmund spaces on $\r^d$]\label{paracomposition_Notations and functional analysis_def Zygmund spaces on r}
For $r\in \r$ we define the space: \[C^r_*(\r^d) \subset \sr'(\r^d),\text{ by }\
C^r_*(\r^d)=\set{u\in\sr'(\r^d),\norm{u}_r=\sup_q 2^{qr}\norm{u_q}_\infty<\infty}\]
 equipped with its canonical topology giving it a Banach space structure.\\
 It's a classical result that for $r\notin \n$, $C^r_*(\r^d)=W^{r,\infty}(\r^d)$ the classic H{\"o}lder spaces.\\
 We define the local spaces:
  \[C^r_{*,loc}(\r^d)=\set{u\in\sr'(\r^d),\forall \psi \in C^\infty_0(\r^d), \psi u\in C^r_*(\r^d)}.\]
\end{definition}

\begin{proposition} \label{paracomposition_Notations and functional analysis_proposition Zygmund spaces on balls}
Let $\b$ be a ball with center 0. There exists a constant C such that for all $r>0$ and for all $(u_q)_{q\in \n}\in \sr'(\r^d)$ verifying:
\[\forall q ,\supp  \hat{u_q} \subset  2^q \b  \text{ and }  (2^{qr}\norm{u_q}_\infty)_{q\in \n} \text{ is  bounded} \]
\[\text{then}, u=\sum_q u_q \in C^r_*(\r^d) \text{ and } \norm{u}_{r} \leq \frac{C}{1-2^{-r}} sup_{q \in \n}2^{qr}\norm{u_q}_\infty. \]
\end{proposition}

For the definition of spaces in open subsets of $\r^d$ we follow the presentation of \cite{Chandler17}. Let $\Omega$ be an open subset of $\r^d$.

\begin{definition}[Zygmund spaces on $\Omega$]
For $r\in \r$ we define the space: \[C^r_*(\Omega) \subset \dr'(\Omega),\text{ by }\ C^r_*(\Omega)=\set{u\in\dr'(\Omega),u=U_{|\Omega} \ for \ some  \ U \in C^r_*(\r^d)}\]
 equipped with its canonical topology that is
 \[\norm{u}_{C^r_*(\Omega)}=\inf_{\substack{U \in C^r_*(\r^d)\\ U_{|\Omega}=u}}\norm{U}_{C^r_*(\r^d)}\]
  giving it a Banach space structure.\\
   We define the local spaces:
  \[C^r_{*,loc}(\Omega)=\set{u\in\dr'(\Omega),\forall \psi \in C^\infty_0(\Omega), \psi u\in C^r_*(\Omega)} .\]
\end{definition}

\begin{definition}[Sobolev spaces on $\r^d$]\label{paracomposition_Notations and functional analysis_def Sobolev spaces on r}
It is also a classical result that for $s\in \r$ :
\[H^s(\r^d)=\set{u\in\sr'(\r^d),\abs{u}_s= \bigg(\sum_q 2^{2qs} {\norm{u_q}_{L^2}}^2 \bigg)^{\frac{1}{2}}<\infty}\]
 with the right hand side equipped with its canonical topology giving it a Hilbert space structure and $\abs{\ }_s$ is equivalent to the usual norm on $\norm{\ }_{H^s}$ .\\
 We define the local spaces:
 \[H^s_{loc}(\r^d)=\set{u\in\sr'(\r^d),\forall \psi \in C^\infty_0(\r^d), \psi u\in H^s(\r^d)}.\]
\end{definition}

\begin{proposition} \label{paracomposition_Notations and functional analysis_proposition Sobolev spaces on balls}
Let $\b$ be a ball with center 0. There exists a constant C such that for all $s>0$ and for all $(u_q)_{\in \n}\in \sr'(\r^d)$ verifying:
\[\forall q ,\supp \hat{u_q} \subset  2^q \b \text{ and } (2^{qs}\norm{u_q}_{L^2})_{q\in \n} \text{ is in} \ L^2(\n) \]
\[\text{then}, u=\sum_q u_q \in H^s(\r^d) \text{ and } \abs{u}_s \leq \frac{C}{1-2^{-s}} \bigg(\sum_q 2^{2qs} {\norm{u_q}_{L^2}}^2 \bigg)^{\frac{1}{2}}. \]
\end{proposition}

The previous definition and properties of the Littlewood-Paley decomposition, Zygmund spaces and Sobolev spaces carries out naturally to $\t^d$.

\begin{definition}[Sobolev spaces on $\Omega$]
For $s\in \r$ we define the space \[H^s(\Omega) \subset \dr'(\Omega),\text{ by }
\ H^s(\Omega)=\set{u\in\dr'(\Omega),u=U_{|\Omega} \ for \ some  \ U \in H^s(\r^d)},\]
 equipped with its canonical topology that is,
 \[\norm{u}_{H^s(\Omega)}=\inf_{\substack{U \in H^s(\r^d)\\ U_{|\Omega}=u}}\norm{U}_{H^s(\r^d)},\]
  giving it a Hilbert space structure\footnote{This is not immediate from the definition but is a consequence of the fact that $H^s(\Omega)$ can be seen as a quotient of $H^s(\r^d)$ by a closed subset, for a full presentation see \cite{Chandler17}.}.\\
   We define the local spaces:
  \[H^s_{loc}(\Omega)=\set{u\in\dr'(\Omega),\forall \psi \in C^\infty_0(\Omega), \psi u\in H^s(\Omega)}. \]
\end{definition}

This definition of the functions in an open subset might not seem as the most natural, in fact there are different ways(intrinsically, extrinsically, by interpolation etc...) to define $H^s(\Omega)$ and when no regularity assumption is put on $\Omega$ and they don't necessarily match. In \cite{Chandler17} they show that when $\Omega$ has Lipschitz regularity all the different definitions of $H^s(\Omega)$ coincide.

We recall the usual nonlinear estimates in Sobolev spaces:
\begin{itemize}
\item If $u_j\in H^{s_j}(\r^d), j=1,2$, and $s_1+s_2>0$ then $u_1u_2 \in H^{s_0}(\r^d)$ and if
\[ s_0\leq s_j, j=1,2 \text{ and } s_0\leq s_1+s_2-\frac{d}{2}, \ \ \  \]
\[\text{then }  \norm{u_1u_2}_{H^{s_0}}\leq K \norm{u_1}_{H^{s_1}}\norm{u_2}_{H^{s_2}} ,\]
where the last inequality is strict if $s_1$ or $s_2$  or $-s_0$ is equal to $\frac{d}{2}$.
\item For all $C^\infty$ function F vanishing at the origin, if $u \in H^s(\r^d)$ with $s>\frac{d}{2}$, then,
\[ \norm{F(u)}_{H^s} \leq C(\norm{u}_{H^s}),\]
for some non decreasing  function C depending only on F.
\end{itemize}
Finally we present a classic result for operator estimates by Y.Meyer \cite{Meyer81}:
\begin{lemma}[Meyer multipliers]\label{paracomposition_Notations and functional analysis_lemme Meyer multiplier} 
Let $\delta \in \r$, and suppose we have a sequence: \[ m_p\in C^{\infty}, \ \forall k\in \n, \ \sum_{\abs{\alpha}=k} \norm{\partial^\alpha m_p}_{\infty}\leq C_k 2^{p(k+\delta)}.\]
The mapping $M: u\mapsto \sum m_pu_p=Mu$ maps $H^s$ to $H^{s-\delta}$ and $C^r_*$ to $C^{r-\delta}_*$ for all $s,r>\delta$, with operators norms depending only on the $C_k$ for $k\leq \lfloor s-\delta \rfloor+1$ or $k\leq \lfloor r-\delta\rfloor+1$.
\end{lemma}

Here we recall the usual Kato-Ponce \cite{Kato88} commutator estimates:
\begin{proposition}\label{paracomposition_Notations and functional analysis_KatoPonce commutator estimate}
Consider $s>0$ and $f,g \in H^s$ then 
\[
\norm{[\D^s,f]g}_{L^2}\leq C(\norm{f}_{W^{1,\infty}}\norm{g}_{H^{s-1}}+\norm{f}_{H^s}\norm{g}_{L^\infty}).
\]
\end{proposition}

\section{Notions of microlocal analysis}\label{paracomposition_section Notions of microlocal analysis} 
In this paragraph we start by reviewing classic notations and results about pseudodifferential calculus, Fourier integral operators and paradifferential calculus, which can be found in \cite{Hormander71}, \cite{Hormander97}, \cite{Taylor07}, \cite{Alinhac07} and \cite{Metivier08} as an accessible presentation to the theories and from which we follow the presentation. Moreover we complete this by our study of the support of the composition of two paradifferential operators. 

\subsection{Pseudodifferential Calculus}
We introduce here the basic definitions and symbolic calculus results. We first introduce the classes of regular symbols.

\begin{definition}\label{paracomposition_Notions of microlocal analysis_Pseudodifferential Calculus_def symbol}
Given $m \in \r,0\leq \rho \leq1$ and $0\leq \sigma \leq1$ we denote the symbol class $S^m_{\rho,\sigma}(\d^d \times \hat{\d}^d)$ the set of all $a\in C^\infty(\d^d \times \hat{\d}^d)$ such that for all multi-orders $\alpha,\beta$ we have the estimate:
\[\abs{\partial^{\alpha}_x\partial^{\beta}_\xi a(x,\xi)}\leq C_{\alpha,\beta}(1+\abs{\xi})^{m-\rho \beta+\sigma \alpha}.\]
$S^m_{\rho,\sigma}(\d^d \times \hat{\d}^d)$ is a Fr\'echet space with the topology defined by the family of semi-norms:
\[M^m_{\alpha,\beta}(a)=\sup_{i\leq \abs{\alpha},j\leq \abs{\beta}}\sup_{\d^d \times \hat{\d}^d}\abs{\partial^{i}_x\partial^{j}_\xi a(x,\xi)(1+\abs{\xi})^{\rho j-m-\sigma i}}.\]
Set \[S^{m}(\d^d \times \hat{\d}^d)=S^m_{1,0}(\d^d \times \hat{\d}^d),\]
\[ S^{-\infty}(\d^d \times \hat{\d}^d)=\bigcap_{m\in \r}S^m(\d^d \times \hat{\d}^d) \text{ and } S^{+\infty}(\d^d \times \hat{\d}^d)=\bigcup_{m\in \r}S^m(\d^d \times \hat{\d}^d) \]
equipped with their canonically induced topology. 
\end{definition}

Given a symbol $a\in S^{m}(\d^d \times \hat{\d}^d)$, we define the pseudodifferential operator:
\[\op(a)u(x)=a(x,D)u(x)=(2\pi)^{-n}\int_{\hat{\d}^d}e^{ix.\xi}a(x,\xi)\hat{u}(\xi)d\xi. \]
For $u \in \sr(\d^d)$ we have 
\begin{equation*} 
\begin{split}
\op(a)u(x) &=(2\pi)^{-d}\int_{\hat{\d}^d}e^{ix.\xi}a(x,\xi)\hat{u}(\xi)d\xi \\
 & = (2\pi)^{-d}\int_{\hat{\d}^d}e^{ix.\xi}a(x,\xi)\int_{\d^d}e^{-iy.\xi}u(y)dy d\xi \\
 & = \int_{\hat{\d}^d}\bigg((2\pi)^{-n}\int_{\d^d}e^{i(x-y).\xi}a(x,\xi) d\xi \bigg)u(y)dy.\\
\end{split}
\end{equation*}

Thus giving us the following proposition.

\begin{proposition}
For $a \in S^m(\d^d \times \hat{\d}^d)$, $\op(a)$ has a kernel K defined by 
\begin{equation} \label{paracomposition_Notions of microlocal analysis_Pseudodifferential Calculus_Kernel equation of a pseudo operator}  
K(x,y)=(2\pi)^{-d}\int_{\hat{\d}^d}e^{i(x-y).\xi}a(x,\xi) d\xi=(2\pi)^{-n}\fr_\xi a (x,y-x).
\end{equation}

Which can be inverted to give:
\begin{align}
a(x,\xi)&=\fr_{y\rightarrow \xi}K(x,x-y)=\int_{\d^d}e^{-iy.\xi}K(x,x-y)dy \nonumber \\
&=(-1)^de^{-ix.\xi}\int_{\d^d}e^{iy.\xi}K(x,y)dy.  \label{paracomposition_Notions of microlocal analysis_Pseudodifferential Calculus_Kernel equation of a pseudo operator inverse}  
\end{align}
\end{proposition}

\begin{definition}
Let $m \in \r$, an operator T is said to be of order m if, and only if, for all $\mu \in \r$, it is bounded from $H^\mu(\d^d)$ to $H^{\mu-m}(\d^d)$. 
\end{definition}

\begin{theorem}
If $a \in S^m(\d^d \times \hat{\d}^d)$, then $a(x,D)$ is an operator of order m. Moreover we have the norm estimate:
\[\norm{a(x,D)}_{H^\mu \rightarrow H^{\mu-m}}\leq C M^m_{\mu,m+d/2+1}(a).\]
\end{theorem}

We will now present the main results in symbolic calculus associated to pseudodifferential operators.
\begin{theorem} \label{paracomposition_Notions of microlocal analysis_Pseudodifferential Calculus_theorem symbolic calclus}  
Let $m,m' \in \r^d$, $a \in S^m(\d^d \times \hat{\d}^d)$and $b \in S^{m'}(\d^d \times \hat{\d}^d)$. 
\begin{itemize}
\item Composition: Then $\op(a)\circ \op(b)$ is a pseudodifferential operator of order $m+m'$ with symbol $a \otimes b$ defined by:
\[a \otimes  b(x,\xi)=(2\pi)^{-d}\int_{\d^d \times \hat{\d}^d}e^{i(x-y).(\xi-\eta)} a(x,\eta)b(y,\xi)dyd\eta.\]
Moreover,
\[\op(a)\circ \op(b)(x,\xi)-\op(\sum_{\abs{\alpha}<k}\frac{1}{i^{\abs{\alpha}}\alpha!}(\partial^\alpha_\xi a(x,\xi))(\partial^\alpha_x b(x,\xi))) \ \text{is of order $m+m'-k$} \]
for all $k\in \n$.
\item Adjoint: The adjoint operator of $\op(a)$, that will note $\op(a)^t$ to avoid confusion with the pullback operator defined in this work, is a pseudodifferential operator of order m with  symbol $a^t$ defined by:
\[a^t(x,\xi)=(2\pi)^{-d}\int_{\d^d \times \hat{\d}^d}e^{-iy.\eta} \bar{a}(x-y,\xi-\eta)dyd\eta\]
Moreover,
\[\op(a^t)(x,\xi)-\op(\sum_{\abs{\alpha}<k}\frac{1}{i^{\abs{\alpha}}\alpha!}(\partial^\alpha_\xi \partial^\alpha_x  \bar{a}(x,\xi))) \ \text{is of order $m-k$} \]
for all $k\in \n$.
\end{itemize}

\end{theorem}
\begin{definition}
Let $(a_j)\in S^{m_j}(\d^d \times \hat{\d}^d)$ be a series of symbols with orders $(m_j) \in \r^{\n}$ decreasing to $-\infty$. We say that $a\in S^{m_0}(\d^d \times \hat{\d}^d)$ is the asymptotic sum of $(a_j)$ if
\[\forall k \in \n, a-\sum_{j=0}^k a_j \in S^{m_{k+1}}(\d^d), \]
and in this case we write
\[
a \sim \sum_j a_j.
\]
\end{definition}

We can now write simply:
\[a \otimes  b \sim \sum_{\abs{\alpha}}\frac{1}{i^{\abs{\alpha}}\alpha!}(\partial^\alpha_\xi a(x,\xi))(\partial^\alpha_x b(x,\xi)) \]
and
\[a^t \sim \sum_{\abs{\alpha}}\frac{1}{i^{\abs{\alpha}}\alpha!}(\partial^\alpha_\xi \partial^\alpha_x  \bar{a}(x,\xi)) .\]

\begin{proposition}[Pseudo-local property] \label{paracomposition_Notions of microlocal analysis_Pseudodifferential Calculus_proposition pseudo local property}  
Let $a\in S^m(\d^d \times \hat{\d}^d)$ and let K be its kernel. Then K is $C^\infty$ for $x \neq y$. In particular, for all $u \in \sr' $:
\[\singsupp  a(x,D)  u \subset \singsupp u\]
\end{proposition}
\begin{proof} Let $x \neq y$, $\psi,\theta \in C^\infty_0(\r^d)$,$\psi=1$ near $x$, $\theta=1$ near y and $\supp \psi \cap \supp \theta= \emptyset$. Then $\tilde{K}(x,y)=\psi(x)K(x,y)\theta(y)$ is the kernel of the operator $\psi a \theta$. By Theorem \ref{paracomposition_Notions of microlocal analysis_Pseudodifferential Calculus_theorem symbolic calclus} ,  $\psi a \theta \sim 0$ thus is of order $-\infty$ which finishes the proof. 

\end{proof}
Let $\Omega$ be an open subset of $\r^d$. We will now define the notion of local symbols and operators in an open set.

\begin{definition}[Local operators and symbols]
We define $S^m(\Omega \times \r^d)$ to be the set of $a \in C^\infty(\Omega\times\r^d)$ such that for all multi-orders $\alpha,\beta$ we have the estimate:
\[\abs{\partial^{\alpha}_x\partial^{\beta}_\xi a(x,\xi)}\leq C_{\alpha,\beta}(1+\abs{\xi})^{m-\rho \beta+\sigma \alpha}.\]
$S^m(\Omega^d \times \r^d)$ is a Fr\'echet space with the topology defined by the family of semi-norms:
\[M^m_{\alpha,\beta}(a)=\sup_{i\leq \alpha,j\leq \beta}\sup_{\Omega \times\r^d}\abs{\partial^{i}_x\partial^{j}_\xi a(x,\xi)(1+\abs{\xi})^{\rho j-m-\sigma i}}.\]
We define the local spaces:
 \[S^m_{loc} (\Omega \times \r^d)=\set{a\in C^\infty(\Omega\times\r^d), \forall \psi \in C^\infty_0(\Omega), \psi a \in S^m(\Omega \times \r^d)},\]
 equipped with its canonical topology giving it a Fr\'echet space structure.
\end{definition}
If $a\in S^m (\Omega \times \r^d) \text{ or } S^m_{loc} (\Omega \times \r^d)$, the usual formula 
\[Au(x)=a(x,D)u(x)=(2\pi)^{-d}\int_{\r^d}e^{ix.\xi}a(x,\xi)\hat{u}(\xi)d\xi \]
defines an operator respectively  from $\sr'(\r^d)$,$\er '(\Omega)$ to $ \dr'(\Omega) $, which can be restricted to an operator $\er '(\Omega)\rightarrow  \dr'(\Omega)$ and $C^\infty_0(\Omega) \rightarrow C^\infty(\Omega)$. \\
The link between such operators and the operators obtained by cut-off from global operators is given by the following proposition:

\begin{proposition} \label{paracomposition_Notions of microlocal analysis_Pseudodifferential Calculus_proposition operstors defined by cut-off and operators on domains}  
Let $A:v\rightarrow C^\infty(\Omega)$ be a continuous linear operator such that for all $\psi,\theta \in C^\infty_0(\Omega)$, $\psi A\theta \in \op(S^m)$. Then there exists $a' \in S^m(\Omega \times \r^d)$ with 
A=a'(x,D)+R, where R is an operator with kernel in $C^\infty(\Omega \times \Omega)$.
\end{proposition}

\begin{proof}
Let $(\psi_j)\in C^\infty_0(\Omega)$ be a partition of unity locally finite over $\Omega$. Put $\psi_jA\psi_k=A_{jk}\in \op(S^m) $ then
\[Au=\sum_{j,k}\psi_jA\psi_k=\sum_{\substack{j,k \\ \supp  \psi_j \cap \psi_k \neq \emptyset}}A_{jk}+\sum_{\substack{j,k \\ \supp  \psi_j \cap \psi_k = \emptyset}}A_{jk}.\]
Then 
\[a'=\sum_{\substack{j,k \\ \supp  \psi_j \cap \psi_k \neq \emptyset}} A_{jk} \in S^m(\Omega \times \r^d) \]
because for $\forall \psi \in C^\infty_0(\Omega),\ \psi a'$ is a finite sum by definition of a partition of unity locally finite.\\
The remainder has a kernel:
\[\sum_{\substack{j,k, \\ \supp \psi_j \cap \psi_k = \emptyset}}\psi_j(x)K(x,y)\psi_k(y) \in C^\infty(\Omega \times \Omega) \]
by the pseudo-local property, Proposition \ref{paracomposition_Notions of microlocal analysis_Pseudodifferential Calculus_proposition pseudo local property} .
\end{proof}
We see from the previous definition that there is subtlety with the support of the functions if one want for example to define $A^t$. The following class of local operators clarifies that problem:
\begin{definition}[Properly supported operators]
A continuous linear operator $A:C^\infty_0(\Omega)\rightarrow C^\infty(\Omega)$ is said to be properly supported if, for any compact subset $K \subset \Omega$, there exists a compact subset $K' \subset \Omega$ with:
\[\supp  u \subset K \Longrightarrow \supp  Au \subset K' \text{ and } u=0 \text{ on } K' \Longrightarrow Au=0  \text{ on } K \] 
\end{definition} 
We see that such an operator maps $C^\infty_0$ to $C^\infty_0$ and for example $A^t$ can be extended in a standard way to an operator from $\dr'(\Omega)$ to itself. \\

\begin{proposition}
Let $A=a(x,D)$ where $a \in S^m_{loc}(\Omega \times \r^d)$. There exists an operator R with kernel in $C^\infty(\Omega \times \Omega)$ such that A+R is properly supported.
\end{proposition}

\begin{proof}
This is the same proof as Proposition \ref{paracomposition_Notions of microlocal analysis_Pseudodifferential Calculus_proposition operstors defined by cut-off and operators on domains}  because 
\[\sum_{\substack{j,k, \\ \supp \psi_j \cap \psi_k = \emptyset}}A_{jk}\]
is properly supported.
\end{proof}

\begin{remark} \label{paracomposition_Notions of microlocal analysis_Pseudodifferential Calculus_remark on properly supported operators} 
The previous proposition tells us that for local regularity considerations  we can essentially work with properly supported operators for local symbols (modulo a $C^\infty$ kernel) and by Proposition \ref{paracomposition_Notions of microlocal analysis_Pseudodifferential Calculus_proposition operstors defined by cut-off and operators on domains}  we can do the same for operators obtained by cut-off. 
\end{remark}

\subsection{Fourier Integral Operators}
Here we will give basic definitions and results as presented in part 1 of H{\"o}rmander's \cite{Hormander71}.\\
We wish to define operators of the form : 

\begin{align} \label{paracomposition_Notions of microlocal analysis_Fourier Integral Operators_eq defining op}  
A_\omega u(x)&=\int e^{iS(x,\xi)}a(x,\xi)\hat{u}(\xi)d\xi\\
&=\int e^{i(S(x,\xi)-y.\xi)}a(x,\xi)u(y)dy\nonumber \\
&=\int e^{i\omega(x,y,\xi)}a(x,\xi)u(y)dyd\xi \nonumber
\end{align}

where $u$ is a regular function, $a$ is a symbol and $\omega$ is a given function defining the operator $A$. We can clearly see that for example $\omega=0$ the integral in not defined for symbols with $m\geq -d$, we thus start by the following definition of suitable phase functions:

\begin{definition}
Let $\omega(x,y,\xi) $ be a $C^\infty(\Omega \times \Omega \times \r^d)$ map which is positively homogeneous of degree one with respect to $\xi$. Put:
\[R_{\omega}=\set{(x,y)\in \Omega \times \Omega, \forall \xi \in \r^d\setminus\set{0}, \omega(x,y,\xi) \ \text{has no critical point}}\footnote{$R_\omega$ is clearly open.},\]
and its compliment $C_{\omega}$, which is the projection on $\Omega \times \Omega$ of the conic set (with respect to $\xi$) of:
\[C=\set{(x,y,\xi)\in \Omega \times \Omega\times \r^d\setminus\set{0}, D\omega_\xi(x,y,\xi)=0}.\]
\begin{itemize}
\item Then $\omega$ is called a phase function on $R_{\omega}\times\r^d$.
\item $\omega$ is called a non-degenerate phase function if at any point in $C$, the differentials $D(\frac{\partial \omega}{\partial \xi_j}),j=1,...,d,$ are linearly independent.
\item $\omega$ is called an operator phase function on $R_{\omega}\times\r^d$ if for each fixed $x$ (or $y$) it has no critical point $(y,\xi)$ $(or (x,\xi))$ with $\xi \neq 0$.
\item For $U \subset \Omega$ define $C_\omega U=\set{x,(x,y)\in C_\omega \ for \ some \ y \in U}.$
\end{itemize}
\end{definition}
The main example here are pseudodifferential operators with $\omega(x,y,\xi)=(x-y).\xi$, in that case $C_\omega$ is equal to the diagonal $\set{(x,x),x\in \Omega}$, and we see that all of the previous definitions  naturally apply in this case. \\
The following proposition will give a definition to the weak form of \eqref{paracomposition_Notions of microlocal analysis_Fourier Integral Operators_eq defining op} :
\begin{equation} \label{paracomposition_Notions of microlocal analysis_Fourier Integral Operators_eq defining op weak form} 
<A_\omega u,v>=<op_\omega(a) u,v>=\int e^{i\omega(x,y,\xi)}a(x,y,\xi)u(y)v(x)dxdyd\xi,  \ u,v\in C^\infty_0(\Omega).
\end{equation}

\begin{proposition}
Take a symbol $a\in S^m_{\rho,\sigma}(\Omega \times \Omega \times \r^d)$,$\rho>0,\sigma<1$, and a phase function $\omega$ on $\Omega \times \Omega \times\r^d$ (that is $R_\omega=\Omega \times \Omega$). Then:
\begin{enumerate}
\item The oscillatory integral \eqref{paracomposition_Notions of microlocal analysis_Fourier Integral Operators_eq defining op weak form}  exists and is a continuous bilinear form for the $C^k_0$ topologies on $u,v$ if 
\[m-k\rho<-N, \  m-k(1-\sigma)<-N.\] 
Thus we obtain a continuous linear map $A$ from $C^k_0(\Omega)$ to $\dr'^k(\Omega)$ which has a distribution kernel $K_\omega \in \dr'^k(\Omega \times \Omega)$ given by the oscillatory integral 
\[K_\omega(u)=\int e^{i\omega(x,y,\xi)}a(x,y,\xi)u(x,y)dxdyd\xi, \ u\in C^\infty_0(\Omega\times \Omega).\]

\item If $\omega$ has no critical point $(y,\xi)$ for each fixed x, then \eqref{paracomposition_Notions of microlocal analysis_Fourier Integral Operators_eq defining op}  is defined as an oscillatory integral and we obtain a continuous map $A: C^k_0(\Omega)\rightarrow C(\Omega)$. By differentiation under the integral sign it follows that $A$ is also continuous map from  $C^k_0(\Omega)$ to $C^j(\Omega)$ if 
\[m-k\rho<-N-j, \ m-k(1-\sigma)<-N-j.\] 
\item If $\omega$ has no critical point $(x,\xi)$ for each fixed y, then the adjoint of $A$ is defined and has the properties listed in point 2, so A is a continuous map of $\er'^j(\Omega)$ into $\dr'^k(\Omega)$. In particular $A$ defines a continuous map from $\er'(\Omega) \text{ into }  \dr'(\Omega)$.
\item The oscillatory integral:
\[K_\omega(x,y)=\int e^{i\omega(x,y,\xi)}a(x,\xi)d\xi \text{ defines a } C^\infty(\Omega \times \Omega=R_\omega)  \text{ map},\]
it follows that $A$ is an integral operator with $C^\infty$ kernel, so $A$ is a continuous map of $\er'(\Omega)$ to $C^\infty(\Omega)$.
\item We have the generalization of the pseudo-local property:
\[\singsupp \  op_\omega(a)u=C_\omega \singsupp  u.\]
\end{enumerate}
When $\omega$ is an operator phase function it verifies all the previous properties.
\end{proposition}

\begin{proposition}
Let $\omega(x,y,\xi) $ be a $C^\infty(\Omega \times \Omega \times \r^d)$ map which is positively homogeneous of degree one with respect to $\xi$ and a be a symbol in $ S^m_{\rho,\sigma}(\Omega \times \Omega \times \r^d)$,$\rho>\sigma$ and that either $\omega$ is linear or that $\rho+\sigma=1$. Suppose that $a$ vanishes of infinite order on $C$, that is $\partial^\alpha a=0$ for all $\alpha \in \n^{3d}$, then we have the same results as in the previous proposition with $m$ replaced by $m-\rho+\sigma$.\\
If a just vanishes on $C$ then we can find $b\in S^{m-\delta+\rho}_{\rho,\sigma}(\Omega \times \Omega \times \r^d)$ such that we have the formal equality $op_\omega(a) u=op_\omega(b) u$.
\end{proposition}
As H{\"o}rmander summed up, when $\omega$ is non degenerate the singularities of the distribution $u\rightarrow op_\omega (a) u$ only depend on the Taylor expansion of $a$ on the set $C$.\\ \\

The following proposition, taken from part 2 of \cite{Hormander71},  gives the natural link between pseudodifferential operators and Fourier Integral operators defined by the phase function $\omega(x,y,\xi)=(x-y).\xi$.
\begin{proposition}\label{paracomposition_Notions of microlocal analysis_Fourier Integral Operators_proposition defining amplitude} 
Consider a real number $m$ and a symbol $c\in S^m(\Omega \times \Omega \times \r^d)$, then:
\[a(x,\xi)=\int_{\Omega \times \r^d}c(x,y,\eta)e^{i(x-y).(\eta-\xi)}dyd\eta \in S^m(\Omega \times \r^d) \]
and we have:
\[\forall u \in C^\infty_0(\Omega), op_{(x-y).\xi}(c)u=\op(a)u=(2\pi)^{-d}\int_{\r^d}e^{ix.\xi}a(x,\xi)\hat{u}(\xi)d\xi.\]
Moreover the asymptotic expansion of $a$ is given by:
\[\forall N \in \n, a(x,\xi)-\sum_{\substack{ \abs{\alpha}< N}} \frac{1}{i^{\abs{\alpha}}\alpha!} \partial^\alpha_\xi \partial^\alpha_y c(x,y,\xi)_{|y=x} \in S^{m-N}(\Omega \times \r^d).\]
\end{proposition}

In the previous setting $c$ is often called an amplitude.

We will not give the proof of these propositions here but we will present the fundamental lemma behind those results and the idea behind it. The main problem is to define oscillatory integrals of the form:
\[ \int e^{i\omega(x,\xi)}a(x,\xi)u(x)dxd\xi, \  u\in C^\infty_0(\Omega),\]  
We start by remarking that the integral is absolutely convergent if a is of order $m<-N$. 

\begin{lemma} \label{paracomposition_Notions of microlocal analysis_Fourier Integral Operators_lemme fonda integral oscillante} 
If $\omega$ has no critical point $(x,\xi)$ with $\xi \neq 0$, then one can find a first order differential operator
\[L=\sum_j h_j \frac{\partial}{\xi_j} + \tilde{h}_j \frac{\partial}{x_j} + c \]
with $h_j \in S^0(\Omega \times \r^d) $ and $\tilde{h}_j , c \in S^{-1}(\Omega \times \r^d)$ such that $L^t e^{i\omega}=e^{i\omega}$. \\ 
L is a continuous map from $S^m_{\rho,\sigma}(\Omega \times \Omega \times \r^d)$ to $S^{m-\eps}_{\rho,\sigma}(\Omega \times \Omega \times \r^d)$ where \\ $\eps=min(\rho,1-\sigma)$.
\end{lemma}
Taking a symbol $a $ of order m we compute:
\begin{align*}
\int e^{i\omega(x,\xi)}a(x,\xi)u(x)dxd\xi&=\int e^{i\omega(x,\xi)}La(x,\xi)u(x)dxd\xi \\ &=\int e^{i\omega(x,\xi)}L^ka(x,\xi)u(x)dxd\xi,
\end{align*}
under the hypothesis $\rho>0$ and $\sigma<1$  we have $\eps >0$ and $L^k a \in S^{m-k\eps}_{\rho,\sigma}(\Omega \times \Omega \times \r^d)$, taking $m-k\eps<-N$ and applying the previous remark we see that the integral is then well defined.

\subsection{Paradifferential Calculus}\label{paracomposition_Notions of microlocal analysis_Paradifferential Calculus}
We start by the definition of symbols with limited spatial regularity. Let $\w\subset \sr'$ be a Banach space.
\begin{definition}\label{paracomposition_Notions of microlocal analysis_Paradifferential Calculus_ def para symbol}
Given $\rho \geq 0$ and $m \in \r$, $\Gamma^m_\w(\d^d)$ denotes the space of locally bounded functions $a(x,\xi)$ on $\d^d\times (\hat{\d}^d \setminus 0)$, which are $C^\infty$ with respect to $\xi$ for $\xi \neq 0$ and such that, for all $\alpha \in \n^d$ and for all $\xi \neq 0$, the function $x \mapsto \partial^\alpha_\xi a(x,\xi)$ belongs to $\w$ and there exists a constant $C_\alpha$ such that for all $\eps>0$:
\begin{equation}\label{paracomposition_Notions of microlocal analysis_Paradifferential Calculus_ definition growth xi condition para} 
\forall \abs{\xi}>\eps, \norm{\partial^\alpha_\xi a(.,\xi)}_{\w}\leq C_{\alpha,\eps} (1+\abs{\xi})^{m-\abs{\alpha}}. 
\end{equation}
The spaces $\Gamma^m_\w(\d^d)$ are equipped with their natural Fr\'echet topology induced by the semi-norms defined by the best constants in \eqref{paracomposition_Notions of microlocal analysis_Paradifferential Calculus_ definition growth xi condition para} .

We will essentially work with $\w=W^{\rho,\infty}$ and write $\Gamma^m_\w=\Gamma^m_\rho$.
\end{definition}

For quantitative estimates we introduce as in \cite{Metivier08}:
\begin{definition}\label{paracomposition_Notions of microlocal analysis_Paradifferential Calculus_ definition semi-norms}
For $m\in \r$ and $a \in \Gamma^m_\w(\d)$, we set
\[M^m_\w(a;n)=\sup_{\abs{\alpha}\leq n} \ \sup_{\abs{\xi}\geq\frac{1}{2}}\norm{(1+\abs{\xi})^{m-\abs{\alpha}}\partial^\alpha_\xi a(.,\xi)}_{\w}, \text{ for } n\in \n.\]
For $\w=W^{\rho,\infty},\rho \geq 0$, we write: 
\[
\Gamma^m_{W^{\rho,\infty}}(\d)=\Gamma^m_\rho(\d) \text{ and }
M^m_\rho(a)=M^m_{W^{\rho,\infty}}(a;1+\lfloor \frac{d}{2}\rfloor).
\]
Moreover we introduce the following spaces equipped with their natural Fr\'echet space structure:
\[
C^{\infty}_b(\d)=\cap_{\rho \geq 0}W^{\rho,\infty}, \ \Gamma^m_\infty(\d)=\cap_{\rho \geq 0}\Gamma^m_\rho(\d), \ \Gamma^{-\infty}_\rho(\d)=\cap_{m\in \r}\Gamma^m_\rho(\d) \text{ and,}
\]
\[
\Gamma^{-\infty}_\infty(\d)=\cap_{\rho \geq 0}\cap_{m\in \r}\Gamma^m_\rho(\d). 
\]
\end{definition}

As presented in \cite{Hormander97,Metivier08} the idea to define paradifferential operators is to regularize the symbols by a cutoff $\psi$, for a paradifferential symbol $ a \in \Gamma^{m'}_\rho(\d^d)$ we will then associate a symbol $\sigma^\psi_a \in S^m_{1,1}(\d^d\times \hat{\d}^d )$. All of the results presented above were for the class $S^m_{1,0} \subset S^m_{1,1}$ and don't generalize to $S^m_{1,1}$, even the $L^2$ continuity. Looking more closely to $a\in S^m_{1,1}$ in \cite{Hormander97}, H\"ormander shows that the essential problems that occur are localized in the frequency regions $(\eta,0)$ and $(-\eta,\eta)$ of $\mathscr{F}_x(a)$. Thus the idea in paradifferential calculus is regularisation by a cutoff in the frequency domain with support bounded away from $(\eta,0)$ and $(-\eta,\eta)$ at infinity. Then $\sigma^\psi_a$ will have this extra spectral localization property that will give them the desired properties as in $S^m_{1,0}$.
\begin{definition-proposition}
Take $m \in \r$, $\Sigma^m_\w(\d^d)$ denotes the subclass of symbols $\sigma \in \Gamma^m_\w(\d^d)$ which satisfy the following spectral condition: $$\exists B_1,B_2,B_3,b>0 \text{ such that }B_1B_3>1 \text{ and } B_3 B_2>B_2+B_2,$$
and $\sigma$ verifies
\begin{equation}\label{paracomposition_Notions of microlocal analysis_Paradifferential Calculus_spectral condition} 
\fr_x \sigma(\eta,\xi)=0 \text{ when }
\abs{\eta}>B_1\abs{\xi}+b \text{ or }
\abs{\xi}>B_2\abs{\eta+\xi}+b.
\end{equation}
When $\w=W^{r,\infty}(\d^d)$ we write $\Sigma^m_\w(\d^d)=\Sigma^m_r(\d^d)$, we also note
\[\w \subset L^\infty(\d^d)\Rightarrow
 \Gamma^m_\w (\d^d)\subset  \Gamma^m_0(\d^d),\
  \Sigma^m_\w(\d^d) \subset \Sigma^m_0(\d^d).
\] 
Moreover, by the Bernstein inequalities \eqref{paracomposition_Notations and functional analysis_bernstein1}:
 \[\Sigma^m_0(\d^d) \subset S^m_{1,1}(\d^d).\] 
 More generally, the spectral condition implies that symbols in $\Sigma^m_\w(\d^d)$ are smooth in $x$ too.
\end{definition-proposition}

\begin{remark}\label{paracomposition_Notions of microlocal analysis_Paradifferential Calculus_remark on sigma} 
The interesting fact now is $\Sigma^m_0(\d^d)$ is shown to still enjoy all of the symbolic calculus and continuity properties announced above for $S^m_{1,0}(\d^d)$.
\end{remark}
\begin{definition-proposition}
Consider four strictly positive real numbers $b,(B_i)_{ \tiny 1\leq i\leq 3}$ verifying:
\begin{equation}\label{paracomposition_Notions of microlocal analysis_Paradifferential Calculus_definition restriction cutoff constants} 
B_1B_3>1 \text{ and } B_3 B_2>B_2+B_2.
\end{equation}
Consider $\psi$ a $C^\infty$ function such that:
\begin{enumerate}
\item 
\[
\psi(\eta,\xi)=0 \text{ when }
\abs{\eta}>B_1\abs{\xi}+b \text{ or }
\abs{\xi}>B_2\abs{\eta+\xi}+b,
\]
\[
\text{and }
\psi(\eta,\xi)=1 \text{ when } \abs{\xi}>B_3\abs{\eta}+b,
\]
\item for all $(\alpha,\beta)\in \n^d \times \n^d,$ there is $C_{\alpha_\beta}$, with $C_{0,0}\leq 1$, such that:
\begin{equation}\label{paracomposition_Notions of microlocal analysis_Paradifferential Calculus_definition cutoff growth hypothesis}
\forall(\xi,\eta): \abs{\partial_\xi^\alpha \partial_\eta^\beta \psi(\xi,\eta)}\leq C_{\alpha,\beta} (1+\abs{\xi})^{-\abs{\alpha}-\abs{\beta}}.
\end{equation}
\end{enumerate}
Such a $\psi$ is called an admissible cut-off function for any positive $b,(B_i)_{i\in \set{1,2,3}}$ verifying \eqref{paracomposition_Notions of microlocal analysis_Paradifferential Calculus_definition restriction cutoff constants} . 

The cutoffs defined in the introduction (with the extra gross hypothesis \eqref{paracomposition_Notions of microlocal analysis_Paradifferential Calculus_definition cutoff growth hypothesis}), $(\psi_H^{B})_{B>2}$ by \eqref{paracomposition_introduction_Hormander cutoff}, $(\psi_M^{\eps})$ by \eqref{paracomposition_introduction_Metivier cutoff} and $(\psi^{B_1,B_2,b})_{B>1,b>0}$ by definition \ref{paracomposition_introduction_new cutoff} are all admissible cutoff functions. 
\end{definition-proposition}

Figure \eqref{paracomposition_Notions of microlocal analysis_Paradifferential Calculus_figure5} illustrate the  condition of admissible cutoff functions in the plane $(\xi,\eta)$ when $d=1$.

\begin{definition-proposition}(Regularisation of a symbol)\label{paracomposition_Notions of microlocal analysis_Paradifferential Calculus_definition regularisation of a symbol by cutoff}
Take $m\in \r$, $a \in \Gamma^m_\w$ and $\psi$ an admissible cut-off function. Define $\sigma^\psi_a$ by
\[\fr_x\sigma^\psi_a(\xi,\eta)=\psi(\xi,\eta) \fr_x a(\xi,\eta)
\text{ then } \sigma^\psi_a \in \Sigma^m_\w(\d^d).\]

When $\w=W^{r,\infty}(\Omega)$ we have the following properties:
\begin{enumerate}
\item This association is bounded:
\[M^m_r(\sigma^\psi_a)\leq C M^m_r(a).\]
\item We have $a-\sigma^\psi_a \in \Gamma^m_0$ and $a-\sigma^{\psi^{1,b}}_a \in \Gamma^m_0$, moreover:
\[M^{m-r}_0(\sigma^\psi_a-a)\leq CM^m_r(a).\]
In particular, if $\psi_1$ and $\psi_2$ are two admissible cut-off functions then the difference $\sigma^{\psi_1}_a-\sigma^{\psi_2}_a$ belongs to $\Sigma^{m-r}_0$ and:
\[M^{m-r}_0(\sigma^{\psi_1}_a-\sigma^{\psi_2}_a) \leq CM^m_r(a).\]
\end{enumerate}
\end{definition-proposition}

\begin{figure}[h!]\label{paracomposition_Notions of microlocal analysis_Paradifferential Calculus_figure5}
\begin{tikzpicture}[xscale=1.3,yscale=0.8,dot/.style={circle,inner sep=1pt,fill,label={#1},name=#1},
 extended line/.style={shorten >=-#1,shorten <=-#1},
 extended line/.default=1cm]
\draw [->] (0,-1.5) -- (0,7.7);
\draw [->] (-4,3) -- (4,3);
\node[right] at (0,7.6) {\scriptsize $\xi$};
\node[below right] at (4,3.2) {\scriptsize $\eta$};
\coordinate (Q1) at (0,5);
\node [fill=black,inner sep=1pt,label=0:$b$] at (Q1) {};
\coordinate (Q2) at (1,7.5);
\coordinate (Q3) at (-1,7.5);
\draw[thick,black,pattern=north east lines,pattern color=black] (Q3) -- (Q1)--(Q2);
\node[above] at (Q3) {\scriptsize $\xi = -B_3\eta+b$};
\node[above ] at (Q2) {\scriptsize $\xi = B_3\eta+b$};
\node at (0.8,6.2) {\scriptsize $\psi = 1$};
\coordinate (SQ1) at (0,1);
\node [fill=black,inner sep=1pt,label=0:$-b$] at (SQ1) {};
\coordinate (SQ2) at (1,-1.5);
\coordinate (SQ3) at (-1,-1.5);
\draw[thick,black,pattern=north east lines,pattern color=black] (SQ3) -- (SQ1)--(SQ2);
\node at (0.8,0.2) {\scriptsize $\psi = 1$};
\node[below] at (SQ3) {\scriptsize $\xi = B_3\eta-b$};
\node[below] at (SQ2) {\scriptsize $\xi = -B_3\eta-b$};
\coordinate (0) at (0,3);
\coordinate (SD1) at (-4,7);
\node[left] at (SD1) {\scriptsize $\xi = -\eta$};
\coordinate (SD2) at (4,-1);
\draw (SD1) -- (SD2);
\coordinate (-B) at (-2.5,3);
\node [fill=black,inner sep=1pt,label=-45:$-b$] at (-B) {};
\coordinate (-B') at (-2.5,5.5);
\draw[dotted] (-B) -- (-B');
\coordinate (P1) at (-4,7.5);
\coordinate (P2) at (-4,6.5);
\draw[thick,black,pattern=north west lines,pattern color=red] (P1) -- (-B')--(P2);
\node[left] at (P1) {\scriptsize $\xi = \frac{B_2(\eta+\frac{b}{B_2})}{1-B_2}$};
\node[left] at (P2) {\scriptsize $\xi = \frac{B_2(\frac{b}{B_2}-\eta)}{1+B_2}$};
\node[above right] at (-B') {\scriptsize $\psi = 0$};
\coordinate (B) at (2.5,3);
\node [fill=black,inner sep=1pt,label=-45:$b$] at (B) {};
\coordinate (B') at (2.5,0.5);
\draw[dotted] (B) -- (B');
\coordinate (P3) at (4,-0.5);
\coordinate (P4) at (4,-1.5);
\draw[thick,black,pattern=north west lines,pattern color=red] (P3) -- (B')--(P4);\node[right] at (P4) {\scriptsize $\xi = \frac{B_2(\eta-\frac{b}{B_2})}{1-B_2}$};
\node[right] at (P3) {\scriptsize $\xi = \frac{-B_2(\frac{b}{B_2}+\eta)}{1+B_2}$};
\node[above right] at (B') {\scriptsize $\psi = 0$};
\coordinate (R2) at (4,3.5);
\coordinate (R3) at (4,2.5);
\draw[thick,black,pattern=north west lines,pattern color=red] (R3) -- (B)--(R2);
\node[above] at (B) {\scriptsize $\psi = 0$};
\node[right] at (R2) {\scriptsize $\xi = \frac{\eta}{B_1}-\frac{b}{B_1}$};
\node[right] at (R3) {\scriptsize $\xi = -\frac{\eta}{B_1}+\frac{b}{B_1}$};
\coordinate (SR2) at (-4,3.5);
\coordinate (SR3) at (-4,2.5);
\draw[thick,black,pattern=north west lines,pattern color=red] (SR3) -- (-B)--(SR2);
\node[above] at (-B) {\scriptsize $\psi = 0$};
\node[left] at (SR2) {\scriptsize $\xi = -\frac{\eta}{B_1}-\frac{b}{B_1}$};
\node[left] at (SR3) {\scriptsize $\xi = \frac{\eta}{B_1}+\frac{b}{B_1}$};
\end{tikzpicture}
\caption{Admissible cut-off functions.}
 \end{figure}
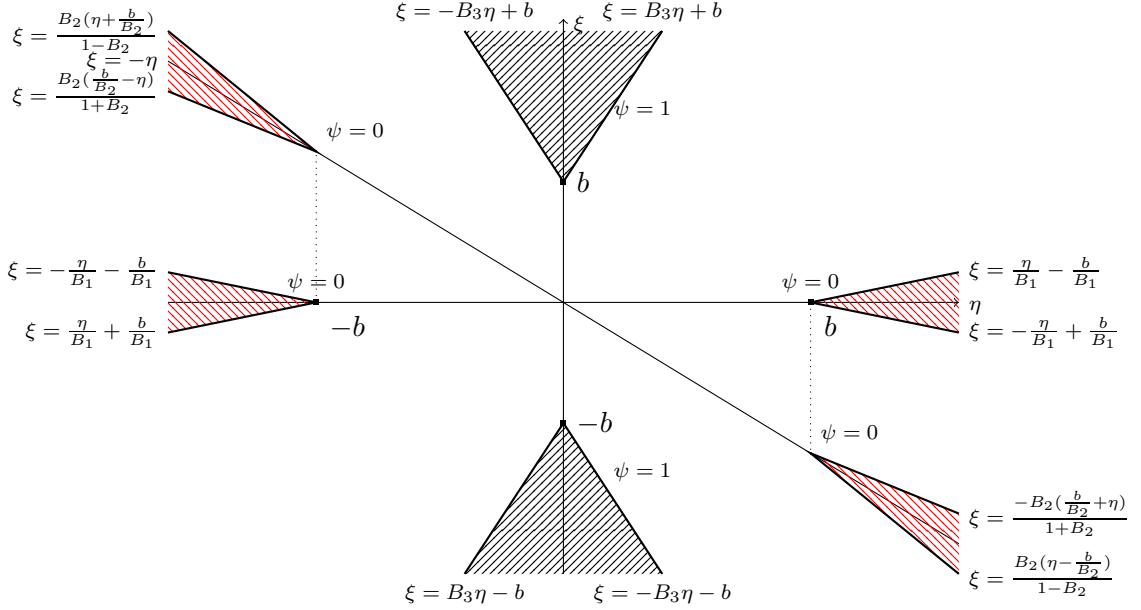

Now we list a couple of important calculus properties to the association $a \mapsto \sigma^{\psi}_a$. For the following proposition we fix a choice of an admissible cutoff function $\psi$ .
\begin{proposition}\label{paracomposition_Notions of microlocal analysis_Paradifferential Calculus_proposition on regularised symbols}
\begin{itemize}
\item For $m \in \r, r\geq 0, \alpha \in \n^d$ of length $\abs{\alpha} \leq r$ and $a\in \Gamma^m_r$:
\[\partial^\alpha_x \sigma^\psi_a=\sigma^\psi_{\partial^\alpha_x a} \in \Sigma^m_0. \]
\item For $m \in \r, r \geq 0$ and $\alpha \in \n^d$ of length $\abs{\alpha}\geq r$ the mapping $a \mapsto \partial^{\alpha}_x \sigma^\psi_a$ is bounded from $\Gamma^m_r$ to $\Sigma^{m+\abs{\alpha}-r}_0$, more precisely:
\[M_0^{m+\abs{\alpha}-r}( \partial^{\alpha}_x \sigma^\psi_a)\leq M^m_r(a).\]
\item For $m \in \r, r \geq 0, \beta \in \n^d$ and $a\in \Gamma^m_r$
\[\partial^{\beta}_\xi \sigma^\psi_a-\sigma^\psi_{\partial^{\beta}_\xi a} \in \Sigma_0^{m-\abs{\beta}-r}.\]
\end{itemize}
\end{proposition}
From \cite{Metivier08} we give an approximation of symbols in $\Sigma^m_0(\d^d)$ by symbols in the Schwartz class.
\begin{lemma}\label{paracomposition_Notions of microlocal analysis_Paradifferential Calculus_lemme approximation regularised symbols}
For all $\sigma \in \Sigma^m_0$ there is a sequence of symbols $\sigma_n \in \sr(\d^d\times \hat{\d}^d)$ such that
\begin{enumerate}
\item the family $\set{\sigma_n}$ is bounded in $S^m_{1,1}$,
\item the $\sigma_n$ satisfy the spectral condition \eqref{paracomposition_Notions of microlocal analysis_Paradifferential Calculus_spectral condition} for some $B_1,B_2,B_3,b>0$ independent of $n$,
\item $\sigma_n \rightarrow \sigma$ on compact subsets of $\d^d\times \hat{\d}^d$.
\end{enumerate}
\end{lemma}

A key property of operators with symbols in $\Sigma^m_0$ is captured in their actions on the spectrum of functions. First we give a general result for symbols in $S^m_{1,1}$ from \cite{Metivier08}.
\begin{proposition}\label{paracomposition_Notions of microlocal analysis_Paradifferential Calculus_propostion pseudo action spectrum}
Consider a real number $m$, $p \in S^m_{1,1}(\d^d\times \hat{\d}^d)$ and $u \in \sr(\d^d)$ then the spectrum of $\op(p)u$ is contained in the closure of the set:
\[\set{\xi+\eta,\xi \in \supp \fr u,(\eta,\xi)\in \supp \fr_x p}.\]
\end{proposition}
This implies the following property for operators verifying the spectral condition \eqref{paracomposition_Notions of microlocal analysis_Paradifferential Calculus_spectral condition}.
\begin{lemma}\label{paracomposition_Notions of microlocal analysis_Paradifferential Calculus_propostion para action spectrum}
Consider a real number $m$, $p \in \Sigma^m_0( \d^d)$ with parameter $B>1,b>0$ and $u \in \sr(\d^d)$.
\begin{itemize}
\item For $R\geq 2 b$, if $\supp \fr u \subset \set{\abs{\xi}\leq R},$ then: 
\begin{equation}\label{paracomposition_Notions of microlocal analysis_Paradifferential Calculus_propostion para action spectrum on rings}
\supp \fr \op(p)u \subset \set{\abs{\xi}\leq (1+\frac{1}{B_1})R-\frac{b}{B}},
\end{equation}
\item For $R\geq 2b$, if $\supp \fr u \subset \set{\abs{\xi}\geq R},$ then: 
\begin{equation}\label{paracomposition_Notions of microlocal analysis_Paradifferential Calculus_propostion para action spectrum on balls}
\supp \fr \op(p)u \subset \set{\abs{\xi}\geq (1-\frac{1}{B_2})R+\frac{b}{B}},
\end{equation}
\end{itemize}
\end{lemma}

The key new result on the control of spectrum of composition and adjoints of paradifferential operators is illustrated in the following.
\begin{proposition}\label{paracomposition_Notions of microlocal analysis_Paradifferential Calculus_prod para spectral localisation}
Take $m,m' \in \r$, and $\rho>0$, $a \in \Gamma^m_\rho(\d^d)$ and $b \in \Gamma^{m'}_\rho(\d^d)$. Consider an admissible cut-off function $\psi^{B,b}$ with $B_1>0,B_2>1$ and $b>$ given by definition \ref{paracomposition_introduction_new cutoff}.
Then we have:
\[
\op\big(\sigma^{\psi^{B_1,B_2,b}}_a\big)^t=\op\bigg(\sigma^{\psi^{\frac{B^2}{2B-1},b}}_{\left(\sigma^{\psi^{B_2-1,B_1+1,b}}_a\right)^t}\bigg),
\]
and for $B_1>1$
\[
\op\big(\sigma^{\psi^{B_1,B_2,b}}_a\big)\circ \op\big(\sigma^{\psi^{B_1,B_2,b}}_b\big)=\op\bigg(\sigma^{\psi^{\frac{B_1^2}{2B_1-1},\frac{B_2^2}{2B_2+1},b}}_{\sigma^{\psi^{B_1,B_2,b}}_a\otimes  \sigma^{\psi^{B_1,B_2,b}}_b}\bigg).
\]

\end{proposition}

\begin{proof}
To understand the adjoint we introduce the following linear operator $$T(\eta,\xi)=(-\eta,\eta+\xi),$$ then we have the following formal identity:
\[
\fr_x\left(\left(\sigma^{\psi^{B_1,B_2,b}}_a\right)^t\right)(\eta,\xi)=[\overline{\fr_x(\sigma^{\psi^{B_1,B_2,b}}_a)}\circ T](\eta,\xi),
\]
we then note that $\psi^{B_1,B_2,b}\circ T=\psi^{B_2-1,B_1+1,b}$ to get the desired result for adjoints.

For the composition we start from the following general identity for the composition of two symbols $\op(p)\circ \op(q)=\op(p\otimes q)$ with $p \otimes  q$ given by
\[p \otimes  q(x,\xi)=(2\pi)^{-d}\int_{\d^d\times \hat{\d}^d}e^{i(x-y).(\xi-\eta)} p(x,\eta)q(y,\xi)dyd\eta.\]
We then compute the Fourier transform in $x$ to get.
\begin{align*}
&\fr_x(p \otimes  q)(\eta,\xi)=(2\pi)^{-d}\int e^{i(x-y).(\xi-\eta_1)} e^{-ix.\eta}p(x,\eta_1)q(y,\xi)dyd\eta_1dx\\
&=(2\pi)^{-3d}\int e^{i(x-y).(\xi-\eta_1)} e^{-ix.\eta}e^{ix.\eta_2}e^{iy.\eta_3}\fr_x(p)(\eta_2,\eta_1)\fr_x(q)(\eta_3,\xi)dxdyd\eta_1d\eta_2d\eta_3\\
&=(2\pi)^{-d}\int\fr_x(p)(\eta-\tilde{\eta},\xi-\tilde{\eta})\fr_x(q)(\tilde{\eta},\xi)d\tilde{\eta},
\end{align*}
where we used $\fr_x(e^{ix.\xi})(\eta)=(2\pi)^{d}\delta_0(\eta-\xi)$.
From the previous formula we get 
\begin{multline*}
\fr_x(\sigma^{\psi^{B_1,B_2,b}}_a\otimes  \sigma^{\psi^{B_1,B_2,b}}_b)(\eta,\xi)\\=(2\pi)^{-d}\int\fr_x\left(\sigma^{\psi^{B_1,B_2,b}}_a\right)(\eta-\tilde{\eta},\xi-\tilde{\eta})\fr_x\left(\sigma^{\psi^{B_1,B_2,b}}_b\right)(\tilde{\eta},\xi)d\tilde{\eta}
\end{multline*}
From the definition is reduced to $d=1$. The strategy is to investigate what happens in the different regions of the plane $(\eta,\xi)$. 

By the definition of $\psi^{B_1,B_2,b}$ and symmetry the problem is reduced to studying the case $d=1$. The strategy is to investigate what happens in the different regions of the plane $(\eta,\xi)$. The zones $\eta\leq 0 ,\xi\geq 0$ and $\eta\geq 0 ,\xi\leq 0$ we already have the``worst" possible estimate on $B_2$ given by the standard case of the H\"ormander cut-offs. The improvement we want to study is for $\eta ,\xi\geq 0$ and $\eta ,\xi\leq 0$ and the estimate on $B_1$. By symmetry it suffices to study the case $\eta ,\xi\geq 0$.

The goal is to investigate if one can find $\eta$ and $\tilde{\eta}$ such that:
\begin{align}
B_1\eta+b&>\xi \label{paracomposition_Notions of microlocal analysis_Paradifferential Calculus_proof para action eq1}\\
 B_1 \abs{\tilde{\eta}}+b&\leq \xi \label{paracomposition_Notions of microlocal analysis_Paradifferential Calculus_proof para action eq2}\\ 
 B\abs{\eta-\tilde{\eta}}+b&\leq \abs{\xi-\tilde{\eta}}\label{paracomposition_Notions of microlocal analysis_Paradifferential Calculus_proof para action eq3}
\end{align}
and in that case find an upper bound on $\eta$.

We have:
\[
 B_1 \abs{\tilde{\eta}}+b\leq \xi < B_1\eta+b \Rightarrow \begin{cases}
 \abs{\tilde{\eta}}<\xi\\
 \abs{\tilde{\eta}}<\eta
 \end{cases}
  \Rightarrow 
  \begin{cases}
\abs{\xi-\tilde{\eta}}=\xi-\tilde{\eta}\\
\abs{\eta-\tilde{\eta}}=\eta-\tilde{\eta}
 \end{cases}
 .
\]
Thus by \eqref{paracomposition_Notions of microlocal analysis_Paradifferential Calculus_proof para action eq3}:
\[
B_1\eta-B\tilde{\eta}+b\leq \xi-\tilde{\eta}\Rightarrow B_1\eta-\xi+b\leq (B_1-1)\tilde{\eta},
\]
for $B_1>1$, we have:
\[
\frac{B_1}{B_1-1}\eta-\frac{\xi}{B_1-1}+\frac{b}{B_1-1}\leq \tilde{\eta},
\]
thus by \eqref{paracomposition_Notions of microlocal analysis_Paradifferential Calculus_proof para action eq2}:
\[
\frac{B_1^2}{B_1-1}\eta-\frac{B_1}{B_1-1}\xi+b\frac{B_1}{B_1-1}\leq \xi-b\Rightarrow \frac{B_1^2}{B_1-1}\eta+b\frac{2B_1-1}{B_1-1}\leq \frac{2B_1-1}{B_1-1}\xi,
\]
which give the desired upper bound:
\[
\frac{B_1^2}{2B_1-1}\eta+b\leq \xi.
\]
\end{proof}

\begin{definition}\label{paracomposition_Notions of microlocal analysis_Paradifferential Calculus_def para op}
Consider a real numbers $m\in \r$, a symbol $a\in \Gamma^m_\w$ and an admissible cutoff function $\psi$ define the paradifferential operator $T_a$ by:
\[\widehat{T_a u}(\xi)=(2\pi)^{-d}\int_{\hat{\d}^d}\psi(\xi-\eta,\eta)\hat{a}(\xi-\eta,\eta)\hat{u}(\eta)d\eta,\]
where $\hat{a}(\eta,\xi)=\int e^{-ix.\eta}a(x,\xi)dx$ is the Fourier transform of $a$ with respect to the first variable.
The connection between two different choices of cut-offs is the following:
\begin{equation}\label{paracomposition_Notions of microlocal analysis_Paradifferential Calculus_difference between 2 choices of cutoff}
\forall a \in \Gamma^m_\rho, \ \sigma^{\psi}_a-\sigma^{\psi'}_a\in \Gamma^{m-\rho}_0.
\end{equation}
 \end{definition}
The first main features of paradifferential operators is their continuity given by the following theorems.
\begin{theorem}\label{paracomposition_Notions of microlocal analysis_Paradifferential Calculus_para continuity}
Take $m \in \r$. If $a\in \Gamma^m_0(\d^d)$, then $T_a$ is of order m. Moreover, for all $\mu \in \r$ there exists a constant K such that:
\[\norm{T_a}_{H^\mu \rightarrow H^{\mu-m}}\leq K M^m_0(a),\text{ and,}\]
\[\norm{T_a}_{W^{\mu,\infty} \rightarrow W^{\mu-m,\infty}}\leq K M^m_0(a), \mu \notin \n.\]
\end{theorem}

The symbolic calculus for paradifferential operators is their continuity given by the following theorem from \cite{Metivier08}.
\begin{theorem} \label{paracomposition_Notions of microlocal analysis_Paradifferential Calculus_symbolic calculus para}
Take $m,m' \in \r$, and $\rho>0$, $a \in \Gamma^m_\rho(\d^d)$and $b \in \Gamma^{m'}_\rho(\d^d)$. 
\begin{itemize}
\item Composition: Then $T_a T_b$ is a paradifferential operator of order $m+m'$ and $T_a T_b- T_{a\#b}$ is of order $m+m'-\rho$ where $a \#b $ is defined by:
\[a \#b=\sum_{\abs{\alpha}<\rho }\frac{1}{i^{\abs{\alpha}}\alpha!} \partial^\alpha_\xi a \partial^\alpha_x b \]
Moreover, for all $\mu \in \r$ there exists a constant K such that
\[ \norm{T_aT_b-T_{a\#b}}_{H^\mu \rightarrow H^{\mu-m-m'+\rho}} \leq K (M^m_\rho (a) M^{m'}_0(b)+M^m_\rho (a) M^{m'}_0(b)). \]

\item Adjoint: The adjoint operator of $T_a$, that we will note $T_a^t$ to again avoid confusion with the pull back operator defined in this work, is a paradifferential operator of order m with  symbol $a^t$ defined by:
\begin{equation}\label{paracomposition_Notions of microlocal analysis_Paradifferential Calculus_definition adjoint para}
a^t=\sum_{\abs{\alpha}<\rho} \frac{1}{i^{\abs{\alpha}}\alpha!}\partial^\alpha_\xi \partial^\alpha_x \bar{a} 
\end{equation}
Moreover, for all $\mu \in \r$ there exists a constant K such that:
\[ \norm{T_a^t-T_{a^t}}_{H^\mu \rightarrow H^{\mu-m+\rho}} \leq K M^m_\rho (a). \]
\end{itemize}
\end{theorem}

Combining Theorem \ref{paracomposition_Notions of microlocal analysis_Paradifferential Calculus_symbolic calculus para} with Proposition \ref{paracomposition_Notions of microlocal analysis_Paradifferential Calculus_prod para spectral localisation} we get the following more precise theorem on composition of paradifferential operators. 

\begin{theorem} \label{paracomposition_Notions of microlocal analysis_Paradifferential Calculus_symbolic calculus para precised}
Take $m,m' \in \r$, and $\rho>0$, $a \in \Gamma^m_\rho(\d^d)$and $b \in \Gamma^{m'}_\rho(\d^d)$. Then for $B_1,B_2>1, b>0$ there exists $r\in \Gamma^{m+m'-\rho}_0(\d^d)$ such that:
\[ M^{m+m'-\rho}_0(r) \leq K (M^m_\rho (a) M^{m'}_0(b)+M^m_\rho (a) M^{m'}_0(b)), \]
and we have for
 \[T^{\psi^{B_1,B_2,b}}_a T^{\psi^{B_1,B_2,b}}_b- T^{\psi^{\frac{B_1^2}{2B_1-1},\frac{B_2^2}{2B_2+1},b}}_{a\#b}=T^{\psi^{\frac{B_1^2}{2B_1-1},\frac{B_2^2}{2B_2+1},b}}_r. \]

\end{theorem}

If $a=a(x)$ is a function of $x$ only, the paradifferential operator $T_a$ is called a para-product. With a good choice of $(B,b)$ in the definition of the cut-off function with respect to our choice of the dyadic decomposition of unity in the Littlewood-Paley decomposition we get that when $a=a(x)$, $T_a$ takes the usual form:
\[ T_au=\sum_{k=1}^\infty \Phi_{k-1}a  u_k .\]
It follows from Theorem \ref{paracomposition_Notions of microlocal analysis_Paradifferential Calculus_symbolic calculus para}and the Sobolev embeddings that:
\begin{itemize}
\item If $a \in H^\alpha(\d^d)$ and $b \in H^\beta(\d^d)$ with $\alpha,\beta>\frac{d}{2}$, then
\[T_aT_b-T_{ab}  \text{ is of order } -\bigg( min\set{\alpha,\beta}-\frac{d}{2} \bigg).\]
\item If $a \in H^\alpha(\d^d)$ with $\alpha>\frac{d}{2}$, then
\[T_a^t-T_{a^t} \text{ is of order } -\bigg(\alpha-\frac{d}{2} \bigg).\]
\end{itemize}
An important feature of para-products is that they are well defined for function $a=a(x)$ which are not $L^\infty$ but merely in some Sobolev spaces $H^r$ with $r<\frac{d}{2}$.
\begin{proposition}
Take $m>0$. If $a\in H^{\frac{d}{2}-m}(\d^d)$ and $u \in H^\mu(\d^d)$ then, \[ T_au \in  H^{\mu-m}(\d^d), \text{ and } \ 
 \norm{T_a u}_{H^{\mu -m}}\leq K \norm{a}_{H^{\frac{d}{2} -m}}\norm{u}_{H^{\mu}}. \]
\end{proposition}

A main feature of para-products is the existence of para-linearisation theorems which allow us to replace nonlinear expressions by paradifferential expressions, at the price of error terms which are smoother than the main terms.

\begin{theorem} \label{paracomposition_Notions of microlocal analysis_Paradifferential Calculus_paralinearisation para product} 
Let $\alpha, \beta \in \r $ be such that $\alpha,\beta> \frac{d}{2}$, then
\begin{itemize}
\item Bony's Linearisation Theorem: for all $C^\infty$ function F, if $a \in H^\alpha (\d^d)$ then
\[ F(a)- F(0)-T_{F'(a)}a \in H^{2\alpha-\frac{d}{2}} (\d^d). \]
\item If $a\in H^\alpha(\d^d)$ and $b\in H^\beta(\d^d)$, then $ab-T_ab-T_ba \in H^{\alpha+ \beta-\frac{d}{2}} (\d^d)$. Moreover there exists a positive constant K independent of a and b such that:
\[\norm{ab-T_ab-T_ba}_{H^{\alpha+ \beta-\frac{d}{2}} }\leq K  \norm{a}_{H^\alpha} \norm{b}_{H^\beta}  .\]
\end{itemize}
\end{theorem}
\subsubsection{Link between Fourier Integral Operators and paradifferential operators} 
In order to give the link between Paradifferential operators and Fourier Integral Operators we start by defining the space of amplitudes for Paradifferential operators.
\begin{definition-proposition}
Take $m \in \r$, $A^m_\w(\r^d)$ denotes the subclass of symbols $c \in \Gamma^m(\w \times \w \times \r^d)$ which satisfy the following spectral condition, $$\exists B_1,B_2,B_3,b>0 \text{ such that }B_1B_3>1 \text{ and } B_3 B_2>B_2+B_2,$$
and $c$ verifies
\begin{equation}\label{paracomposition_Notions of microlocal analysis_Paradifferential Calculus_spectral condition para amplitude}
 \fr_{x,y} c(\xi,\zeta,\eta)=0 \text{ for } B_2\abs{\xi-\zeta}+b>\abs{\eta} \text{ or } B_1\abs{\zeta}+b>\abs{\eta}.
\end{equation}
When $\w=W^{r,\infty}(\Omega)$ we write $A^m_\w(\r^d)=A^m_r(\r^d)$.\\
By the Bernstein inequalities \eqref{paracomposition_Notations and functional analysis_bernstein1}, $A^m_0(\r^d) \subset S^m_{1,1}(\r^d)$. More generally, the spectral condition implies that symbols in $A^m_\w(\r^d)$ are smooth in $x,y$ too.
\end{definition-proposition} 
\begin{proposition}\label{paracomposition_Notions of microlocal analysis_Link between Fourier Integral Operators and paradifferential operators_para amplitude}
Consider two real numbers $m \in \r$, $r \in \r_{+}$ and an amplitude $c\in A^m_r( \r^d)$, then:
\[\sigma(x,\xi)=\int_{\Omega \times \r^d}c(x,y,\eta)e^{i(x-y).(\eta-\xi)}dyd\eta \in \Sigma^m_r( \r^d) \]
and we have:
\[\forall u \in C^\infty_0(\Omega), op_{(x-y).\xi}(c)u=\op(\sigma)u=(2\pi)^{-d}\int_{\r^d}e^{ix.\xi}\sigma(x,\xi)\hat{u}(\xi)d\xi.\]
Moreover the asymptotic expansion of a is given by:
\[\sigma(x,\xi)-\sum_{\substack{ \abs{\alpha}< N}}\frac{1}{i^{\abs{\alpha}}\alpha!} \partial^\alpha_\xi \partial^\alpha_y c(x,y,\xi)_{|y=x} \in \Sigma^{m-N}_{r-N}( \r^d).\]
\end{proposition}
\begin{proof}
First by Lemma \ref{paracomposition_Notions of microlocal analysis_Paradifferential Calculus_lemme approximation regularised symbols} we can work with an amplitude $c$ in $\sr$. As $\sr \subset S^m_{1,0} $
  by Proposition \ref{paracomposition_Notions of microlocal analysis_Fourier Integral Operators_proposition defining amplitude}  we have
  \[\sigma(x,\xi)=\int_{\r^d \times \r^d}c(x,y,\eta)e^{i(x-y).(\eta-\xi)}dyd\eta \in \sr.\]
Moreover writing 
\[\fr_x \sigma(\eta,\xi)=\int_{\r^d}\fr_{x,y}c(\xi+\eta-\tilde{\eta},\tilde{\eta}-\xi,\tilde{\eta})d\tilde{\eta},\]
we see that if $c$ verifies the spectral condition with parameters $B,b$ then so does $\sigma$ with parameter $B-1,b$ thus $\sigma \in \Sigma^m_r( \r^d)$. The asymptotic expansion comes from the one given in Proposition \ref{paracomposition_Notions of microlocal analysis_Fourier Integral Operators_proposition defining amplitude}  combined by the symbolic calculus rules in Proposition \ref{paracomposition_Notions of microlocal analysis_Paradifferential Calculus_proposition on regularised symbols}.
 \end{proof}
\section{Pull-back of pseudo and para- differential operators} \label{paracomposition_section Pull-back of pseudo and para- differential operators}
Let $\Omega, \Omega'$ be two open subsets of $\r^d$. Henceforth we will note all variables in $\Omega'$ with a $' $ for clarity in the computations.
Let $\chi:\Omega \rightarrow \Omega'$ be a $C^\infty$ map, $\chi$ gives rise naturally to the pull back operation for functions and kernels:
\begin{align*}
  &C^\infty(\Omega') \rightarrow C^\infty(\Omega) & & C^\infty(\Omega' \times \Omega') \rightarrow C^\infty(\Omega \times \Omega)       \\
  &v \mapsto v \circ \chi=v^* & &K(x',y') \mapsto K(\chi(x),\chi(y))|\deter D\chi(y)|=K^*(x,y).
\end{align*}
This Pull back has the property:
\begin{equation} 
\begin{split}
K^*v^* & = \int_{\Omega}K(\chi(x),\chi(y))v(\chi(y))|\deter D\chi(y)|dy  \\
 & = \int_{\Omega'}K(\chi(x),y')v(y')\# \chi^{-1}(y')dy' =(K(v\# \chi^{-1}))^*.
\end{split}
\end{equation}

Where $\# \chi^{-1}:\Omega' \rightarrow \bar{\n}$ is the function counting the number if pre-images and $v \in C^\infty_0(\Omega')$. We note that the the change of variables is well defined if and only if one of the two integrals is defined.
If $\chi$ is a diffeomorphism we have the usual fonctorial property $K^*v^*=(Kv)^*$ which permits the definition of operators with kernels on manifolds.\\ The classic result on the change of variables in pseudo-differential operators is that for $A\in S^m_{loc}(\Omega' \times \r^d)$ properly supported with kernel K then the operator defined by $K^*$ is a pseudo-differential operator $A^*$ of order m on $\Omega$ which is also properly supported. Thus it can be seen as the stability of this sub-class of operators of kernels under the pull back by diffeomorphisms  (modulo a $C^\infty$ kernel as in Remark \ref{paracomposition_Notions of microlocal analysis_Pseudodifferential Calculus_remark on properly supported operators} ) and thus are well defined on manifolds by the same process. Before we start by presenting those classic results we will discuss why they are essentially optimal.\\

We start by computing for a pseudo-differential operator defined by $a\in S^m(\Omega' \times \r^d)$ with kernel K and $\chi:\Omega \rightarrow \Omega'$ a $C^\infty$ map:
\begin{align*}
K^*u&=\int_{\Omega}K(\chi(x),\chi(y))u(y)|\deter D\chi(y)|dy\\
&=\int_{\Omega\times \Omega}(2\pi)^{-d}e^{i(\chi(x)-\chi(y)).\xi}a(\chi(x),\xi)u(y)|\deter D\chi(y)|dyd\xi
\end{align*}
thus $$K^*=op_{(\chi(x)-\chi(y)).\xi}(a(\chi(x),\xi)|\deter D\chi(y)|)$$ with $$a(\chi(x),\xi)|\deter D\chi(y)| \in S^m(\Omega \times \Omega \times \r^d),$$ because and all the derivatives of $\chi$  are bounded. Put $$\omega_\chi(x,y,\xi)=(\chi(x)-\chi(y)).\xi,$$ by the definitions on Fourier integral operators we have:
\[C_{\omega_\chi}=\set{(x,y)\in \Omega^2,\chi(x)=\chi(y)}.\]
We also see that $w_\chi$ is non degenerate on $\Omega \times \Omega$ if and only if $\chi$ is a local diffeomorphism. 
To sum up:
\begin{proposition}
Take $a\in S^m(\Omega' \times \r^d)$ and $\chi \in C^\infty(\Omega,\Omega')$. Then the pull-back of $\op(a)$ under $\chi$ is a Fourier Integral Operator with phase function $w_\chi$ and symbol $a(\chi(x),\xi)|\deter D\chi(y)| \in S^m(\Omega \times \Omega \times \r^d)$. We have:
\[C_{\omega_\chi}=\set{(x,y)\in \Omega^2,\chi(x)=\chi(y)}.\]
Moreover, $w_\chi$ is non-degenerate if and only of $\chi$ is a local diffeomorphism.
\end{proposition}
Now we ask the question if there exists a symbol $a^*$ such that: \[op_{\omega_\chi}(a(\chi(x),\xi)|\deter D\chi(y)|)=\op(a^*).\]
The classic result is that this true if $\chi$ is a diffeomorphism. Now we precise that it's essentially optimal as it could be seen by the following two examples: \\
\begin{itemize}
\item The necessity of the injectivity of $\xi$: we take $\chi=\abs{\ }$ which is a local diffeomorphism from $\r \setminus 0$ in to $\r^+_*$. We compute for $A=Id$ that is $a=1$:
\[op_{\omega_\chi}(a(\chi(x),\xi)|\deter D\chi(y)|)u=u(x)+u(-x),\]
and the part $u(.) \mapsto u(-.) $ is not a pseudo-differential operator.
\item The necessity of the local diffeomorphism hypothesis: we take $\chi=x^3$ which is a local diffeomorphism from $\r \setminus 0$ in to $\r$. We compute for $A=\frac{d}{dx} $ that is $a=i\xi$:
\[op_{\omega_\chi}(a(\chi(x),\xi)|\deter D\chi(y)|)u=\frac{u'(x)}{3x^2},\]
which is a pseudo-differential operator on $\r \setminus 0$ but cannot be extended to one on $\r$  with a regular symbol in 0. \footnote{In fact it can be treated in the more general frame of operators with singular symbols but this goes beyond the scope of this work.}
\end{itemize}
Now we present the classic results of change of variables in pseudo and para-differential operators under the hypothesis that $\chi$ is a diffeomorphism as they can be found in \cite{Alinhac86},\cite{Alinhac07} and \cite{Hormander71}.

\begin{theorem} \label{paracomposition_section Pull-back of pseudo and para- differential operators_theorem change of variable pseudo}
Let $\chi:\Omega \rightarrow \Omega'$ be a $C^\infty$ diffeomorphism and $A=a(x,D) \in S^m_{loc}(\Omega' \times \r^d)$ a properly supported pseudo-differential operator with kernel K.\\
Then the operator $A^*$  defined by $K^*$ that is:
\[\forall u \in C^\infty_0(\Omega), A^*u =\int_{\Omega}K(\chi(x),\chi(y))u(y)|\deter D\chi(y)|dy \] 
is a properly supported pseudo-differential operator with symbol
\[a^*(x,\xi)=(-1)^de^{-ix.\xi}\int_{\Omega \times \r^d} a(\chi(x),\eta) e^{i(\chi(x)-\chi(y)).\eta+iy.\xi}|\deter  D\chi(y)|dyd\eta \in  S^m_{loc}(\Omega \times \r^d).\]

An expansion of $a^*$ is given by:
\begin{equation}\label{paracomposition_section Pull-back of pseudo and para- differential operators_theorem change of variable pseudo eq1}
a^*(x,\xi) \sim \sum_{\alpha}\frac{1}{\alpha!}\partial^\alpha a(\chi(x),D\chi^{-1}(\chi(x))^t\xi)P_{\alpha}(\chi(x),\xi),
\end{equation}
where,
\[ P_{\alpha}(x',\xi)=D^\alpha_{y'}(e^{i(\chi^{-1}(y')-\chi^{-1}(x')-D \chi^{-1}(x')(y'-x')).\xi})_{|y'=x'} \]
and $P_{\alpha}$ is polynomial in $\xi$ of degree $\leq \frac{\abs{\alpha}}{2}$, with $P_{0}=1, P_{1}=0$.
\end{theorem}
\begin{remark}\label{paracomposition_section Pull-back of pseudo and para- differential operators_remark connection to regularised symbols}
This a classic result found commonly in the literature, And as in the Remark \ref{paracomposition_Notions of microlocal analysis_Paradifferential Calculus_remark on sigma}  an analogous result still holds in the class $\Sigma^m_0$ as will be shown in the proof of the next theorem.
\end{remark}
For para-differential operators we have:
\begin{theorem}  \label{paracomposition_section Pull-back of pseudo and para- differential operators_theorem change of variable para} 
Let $\chi:\Omega \rightarrow \Omega'$ be a $W^{1+\rho,\infty}_{loc}$ diffeomorphism with $D\chi \in W^{\rho,\infty}$ and $\rho \geq 0$. Consider $a \in \Gamma^m_r(\r^d)$ a properly supported paradifferential operator.\\
Then there exists a property supported $a^*\in  \Gamma^m_{min(r,\rho)}(\r^d)$ defined by:
\[(T_a u)\circ \chi=T_{a^*}(u\circ \chi)+(R\chi)u+Ru,\]
where $R\in \Gamma^{m-\min(r,\rho)}_0(\r^d)$ and $R\chi$ is a term depending essentially on $\chi$ and it's explicit formula is given in \eqref{reste changement de variable}.

Moreover $a^*$ has the local expansion:
\begin{equation}\label{paracomposition_section Pull-back of pseudo and para- differential operators_theorem change of variable para eq1}
a^*(x,\xi) \sim \sum_{\substack{ \alpha \\ \abs{\alpha}\leq \lfloor min(r,\rho) \rfloor}}\frac{1}{\alpha!}\partial^\alpha a(\chi(x),D\chi^{-1}(\chi(x))^t\xi)P_{\alpha}(\chi(x),\xi),
\end{equation}
where,
\[ P_{\alpha}(x',\xi)=D^\alpha_{y'}(e^{i(\chi^{-1}(y')-\chi^{-1}(x')-D \chi^{-1}(x')(y'-x')).\xi})_{|y'=x'} \]
and $P_{\alpha}$ is polynomial in $\xi$ of degree $\leq \frac{\abs{\alpha}}{2}$, with $P_{0}=1, P_{1}=0$.
\end{theorem}
 An analogous result still holds for para-differential operators modeled on the spaces $a\in C^r_*,r>0$ and $\chi \in C^{1+\rho}_*$.\\
 As we couldn't find a clear reference to this result in the literature, it is eluded to in \cite{Alinhac86}\footnote{part 3.3 point h, which can be found in pages 114-115.}, we give a simple proof of this theorem.
\begin{proof}
Taking $\psi $ a cut-off function with parameters $B>1, b>0$, and take $u\in C^\infty_0(\Omega)$ compute
\begin{align*}
(T_a (u \circ \chi^{-1})) \circ \chi &= op_{(\chi(x)-\chi(y)).\xi}(\sigma^\psi_a(\chi(x),\xi)\abs{\deter  D\chi(y)})u\\
			&=\int_{\Omega \times \r^d} e^{i(\chi(x)-\chi(y))\cdot \xi}\sigma^\psi_a(\chi(x),\xi)\abs{\deter  D\chi(y)}u(y)dy d\xi.
\end{align*}
As we remarked above the main contribution in this integral will come from $(x,y,\xi) \in C_{\omega_\chi}$ where we recall $\omega_\chi(x,y,\xi)=(\chi(x)-\chi(y))\cdot \xi$. To show this insert the smooth cut-off function $\theta(x,y)$ supported in a small neighborhood of the diagonal $(x,x)$.
\begin{align*}
(T_a (u \circ \chi^{-1})) \circ \chi &=\int_{\Omega \times \r^d} e^{i(\chi(x)-\chi(y))\cdot \xi}\sigma^\psi_a(\chi(x),\xi)\abs{\deter  D\chi(y)}u(y)dy d\xi\\
		&=\int_{\Omega \times \r^d} e^{i(\chi(x)-\chi(y))\cdot \xi}\theta(x,y)\sigma^\psi_a(\chi(x),\xi)\abs{\deter  D\chi(y)}u(y)dy d\xi\\
		&+\int_{\Omega \times \r^d} e^{i(\chi(x)-\chi(y))\cdot \xi}(1-\theta(x,y))\sigma^\psi_a(\chi(x),\xi)\abs{\deter  D\chi(y)}u(y)dy d\xi\\
\end{align*}
Now $\omega_\chi$ has no critical points on the support of $(1-\theta(x,y))$ and by integration by parts we have:
\begin{align*}
(T_a (u \circ \chi^{-1})) \circ \chi &=\int_{\Omega \times \r^d} e^{i(\chi(x)-\chi(y))\cdot \xi}\theta(x,y)\sigma^\psi_a(\chi(x),\xi)\abs{\deter  D\chi(y)}u(y)dy d\xi+Ru.
\end{align*}
with $R \in \Gamma^{m-\min(r,\rho)}_0.$ We now analyze when $y$ is close to $x$. By the mean value Theorem, for $y$ sufficiently close to $x$, there exists a invertible linear mapping $L_{x,y} \in W^{\rho,\infty}$ such that
\[
\begin{cases}
\chi(x)-\chi(y)=L_{x,y}\cdot(x-y)\\
L_{x,x}=D\chi(x).
\end{cases}
\]
Thus we get,
\begin{align*}
&(T_a (u \circ \chi^{-1})) \circ \chi \\ 
&=\int_{\Omega \times \r^d} e^{i(\chi(x)-\chi(y))\cdot \xi}\theta(x,y)\sigma^\psi_a(\chi(x),\xi)\abs{\deter  D\chi(y)}u(y)dy d\xi+Ru\\
		&=\int_{\Omega \times \r^d} e^{i(x-y)\cdot \xi}\theta(x,y)\sigma^\psi_a(\chi(x),{L^t_{x,y}}^{-1}\xi)\abs{\deter  D\chi(y)}\abs{\deter L^{-1}_{x,y}}u(y)dy d\xi+Ru.
\end{align*}
We get an operator with an amplitude
\[c(x,y,\xi)=\theta(x,y)\sigma^\psi_a(\chi(x),{L^t_{x,y}}^{-1}\xi)\abs{\deter  D\chi(y)}\abs{\deter L^{-1}_{x,y}} \in \Gamma^m_{\rho}(\r^d).\]
In the frequency domain this amplitude depends on terms coming from $\sigma^\psi_a(\chi(x),{L^t_{x,y}}^{-1}\xi)$, $\abs{\deter  D\chi(y)}$ and $\abs{\deter L^{-1}_{x,y}}$. 
Define the symbol $b$ associate to the amplitude $c$ by
\[
b(x,\xi)=\int_{\Omega \times \r^d}c(x,y,\eta)e^{i(x-y).(\eta-\xi)}dyd\eta,
\]
and define
 $$
K=\max(1,\sup D\chi^{-1},\sup D\chi), K'=K=\max(1,\sup D\chi)
 $$
 Suppose that the admissible cutoff $\psi$ defining $T_a=\op(\sigma^\psi_a)$ is given with parameters $B_1,B_2,B_3$, then define $\psi'$ to be the cutoff with the new parameters $\frac{B_1}{K'},\frac{B_2}{K}$ and $B_3'$ chosen accordingly to verify the desired constraints. Note that we can always choose $B_1,B_2$ and $B_3$ sufficiently large initially so that $\frac{B_1}{K'}$ and $\frac{B_2}{K}$ are still large enough to ensure the existence of $B_3'$.
 
 Now we compute 
 \[
 (T_a (u \circ \chi^{-1})) \circ \chi=\op(b)u+Ru=T^{\psi'}_b+R\chi+Ru,
 \]
 where $R\chi=\op(b-\sigma^{\psi'}_b)$. Note that all the high frequency terms depending on $\chi$ and $\chi^{-1 }$ are now in the term $R\chi$ by definition of $\psi'$ ( a detailed analysis if which is given in \eqref{reste changement de variable}). The result then follows from Proposition \ref{paracomposition_Notions of microlocal analysis_Link between Fourier Integral Operators and paradifferential operators_para amplitude} applied to $T_b$.
\end{proof}
\section{Paracomposition}\label{paracomposition_section Paracomposition}

\subsection{Main results for paracomposition on $\r^d$} \label{paracomposition_section Paracomposition_subsec paracomp on the euclidean space} 
We start by a formal computation, as in \cite{Taylor07}, using the Littlewood-Paley decomposition and two functions $u$ and $\chi$:
\begin{align}
u \circ \chi &= \sum_{k\geq 0}u(\Phi_{k+1} \chi)-u(\Phi_k \chi)=\sum_{j,k}u_j(\Phi_{k+1} \chi)-u_j(\Phi_k \chi)  \nonumber \\
		&=\sum_{j<k}u_j(\Phi_{k+1} \chi)-u_j(\Phi_k \chi)+\sum_{j\geq k}u_j(\Phi_{k+1} \chi)-u_j(\Phi_k \chi) \\
		&=\underbrace{\sum_{k \geq 1}\Phi_{k-1}u(\Phi_{k} \chi)-\Phi_{k-1}u(\Phi_{k-1} \chi)}_1+\underbrace{\sum_{k\geq 0}u_k(\Phi_{k} \chi)}_2. \nonumber
\end{align}

Heuristically the term 1 has frequencies of $u$  smaller than that of $\chi$ and as in classical paradifferential results will depend mainly on the regularity of $\chi$. This is indeed the main term in Bony's para-linearisation theorem modulo a more regular remainder:

\begin{align}
(1)    &=\sum_{k \geq 1}\bigg(\int_0^1 \Phi_{k-1}u'(\tau \Phi_k\chi+(1-\tau)\Phi_{k-1}) \chi d\tau\bigg) \phi_k\chi  \nonumber \\
	&=\underbrace{\sum_{k \geq 1} \Phi_{k-1}(u' \circ \chi)(\phi_{k} \chi)}_{T_{u'\circ \chi}\chi}\\ &+ \underbrace{\sum_{k \geq 1} \bigg(\int_0^1 \Phi_{k-1}u'(\tau \Phi_k\chi+(1-\tau)\Phi_{k-1} \chi)- \Phi_{k-1}(u' \circ \chi)d\tau\bigg) \phi_k\chi}_{R_0}.\nonumber
\end{align}
Same as term 1, heuristically term 2 will essentially depend on the regularity of u, with a remainder depending on $\chi$ and $u$ that is more regular when it's well defined. Thus (2) will naturally give rise to the paracomposition operator. To better understand it, let us suppose just for the next computation that $\chi$ is linear and invertible:
 
\begin{align*}
(2)    &=\sum_{k \geq 0} \int_{\r^d}\phi_k(\xi)\hat{u}(\xi)e^{i\Phi_k \chi(x).\xi}d\xi  \\
	&=\sum_{k \geq 0} \int_{\r^d}\phi_k(\Phi_k\left(\chi^{-1}\right)^t\xi)\hat{u}(\Phi_k\left(\chi^{-1}\right)^t\xi)e^{ix.\xi}|\Phi_k\left(\chi^{-1}\right)^t(\xi)|d\xi
\end{align*}
Thus we essentially have to look at how $\Phi_k\left(\chi^{-1}\right)^t$ modifies the frequencies and thus how it modifies the rings in the Littlewood-Paley decomposition.\\ Put $\set{k\geq 1, C_k'=\supp \ \phi_k(\Phi_k\left(\chi^{-1}\right)^t.)}$, we have:
 \[C_k'\approx \bigcup_{k-N'\leq l \leq k+N}C_l,\]
 where N and N' are such that $2^N>sup_{k,\r^d} \abs{\Phi_k\chi'}$ and $2^{N'}>sup_{k,\r^d} \abs{\Phi_k\chi'}^{-1}$ and the natural para-composition operator in this case is obtained by cutting the frequencies according to $C'_k$, this is exactly the ``lemme de recoupe" in Alinhac's work.
 
 Now we define $N$ as in the previous remark and compute:
 \begin{align}
(2)    &= \underbrace{\sum_{k\geq 0}  \sum_{\substack{l\geq 0 \\ l\leq k+N}}\phi_l(D)(u_k\circ \chi}_{\chi^\star u })\\ & +\underbrace{\sum_{k\geq 0} \sum_{\substack{l\geq 0 \\ l\leq k+N}}\phi_l(D)[u_k\circ\Phi_{k} \chi-u_k\circ \chi]}_{R_1}+\underbrace{\sum_k(Id-\Phi_{k+N})(D)u_k\circ \Phi_k \chi }_{R_2}.\nonumber
\end{align}

Henceforth for $u\in \mathcal{S}'$ we define
\[
\chi^\star u=\sum_{k\geq 0}  \sum_{\substack{l\geq 0 \\ l\leq k+N}}\phi_l(D)u_k\circ \chi,
\]
where $N$ is chosen as above.

 \begin{theorem} \label{paracomposition_section Paracomposition_subsec paracomp on the euclidean space_theorem defintion of paracomposition}
 Let $\chi:\r^d \rightarrow \r^d$ be a $C^{1+\rho}_{*,loc}$ map with $D\chi \in C^{\rho}_*$ and $\rho>0$ \footnote{Clearly when there is no diffeomorphism hypothesis on $\chi$ we can choose $\chi:\r^d \rightarrow \r^{d'}$ with $d\neq d'$ and have the same results but for clarity we chose to present the same dimensions.}.  Then for all $\sigma,s \in \r^*_+$ the following maps are extended continuously: 
\[ \chi^\star:C^\sigma_{*}(\r^d) \rightarrow  C^\sigma_{*}(\r^d) \text{ and } \chi^\star: C^\sigma_{*,loc}(\r^d) \rightarrow  C^\sigma_{*,loc}(\r^d).
\]
  If moreover $\chi$ is a diffeomorphism then we have the Sobolev estimates:  
  \[ \chi^\star:H^s(\r^d) \rightarrow H^s(\r^d) \text{ and } \chi^\star: H^s_{loc}(\r^d) \rightarrow H^s_{loc}(\r^d).
\]
Taking $\tilde{\chi}:\r^d \rightarrow \r^d$ a $C^{1+\tilde{\rho}}_{*,loc}$ map with $D\tilde{\chi} \in C^{\tilde{\rho}}_*$ and $\tilde{\rho}>0$, then the previous operation has the natural fonctorial property:
\[ \forall u \in C^{\sigma}_{*}(\r^d)\cup C^{\sigma}_{*,loc}(\r^d) , \chi^\star \tilde{\chi}^\star u= ({\chi \circ \tilde{\chi}})^\star u +Ru.\]   
 \[\text{with } R, \  R:C^{\sigma}_{*}(\r^d) \rightarrow  C^{\sigma+min(\rho,\tilde{\rho})}_{*}(\r^d), \ R:C^{\sigma}_{*,loc}(\r^d) \rightarrow  C^{\sigma+min(\rho,\tilde{\rho})}_{*,loc}(\r^d),\]
 and if $\chi$ and $\tilde{\chi}$ are diffeomorphisms:
 \[ \ \ \ \  R:H^{s}(\r^d) \rightarrow H^{s+min(\rho,\tilde{\rho})}(\r^d), \ \ \ \ R:H^{s}_{loc}(\r^d) \rightarrow H^{s+min(\rho,\tilde{\rho})}_{loc}(\r^d).\]
  \end{theorem}

  It is natural that the Sobolev estimates only hold when $\chi$ is a diffeomorphism because for example even the usual composition operation $u \mapsto u \circ \chi$ is not necessarily continuous on $L^p$ spaces, $p<\infty$. An extra hypothesis that appears in the literature is $\chi$ is a local diffeomorphism with all of it is local inverses uniformly bounded in $\dot{W}^{1,\infty}$.
  
 \begin{theorem}  \label{paracomposition_section Paracomposition_subsec paracomp on the euclidean space_theorem paralinearisation of composition}
 Let $u$ be a $W^{1,\infty}(\r^d)$ map and  $\chi$ be a $C^{1+\rho}_{*,loc}$ map with $D\chi \in C^{\rho}_*$ and $\rho>0$ . Then:
 \[u \circ \chi(x)=\chi^\star u(x)+ T_{u'\circ \chi}\chi(x)+ R_0(x)+R_1(x)+R_2(x)\]
 where the paracomposition given in the previous theorem verifies the estimates:
 \[\forall \sigma > 0, \norm{\chi^\star u(x)}_{\sigma}\leq C(\norm{D\chi}_{\infty})\norm{u(x)}_{\sigma},\]
 \[u'\circ \chi \in  \Gamma^0_{W^{0,\infty}(\r^d)}(\r^d) \ \ for\  u \text{ Lipchitz,}   \]
and the remainders verify the estimates: 
\begin{itemize}
\item In Zygmund Spaces, for $\sigma>0$:
\[\norm{R_0}_{1+\rho +min(1+\rho,\sigma)} \leq C \norm{D\chi}_{\rho}\norm{u}_{1+\sigma}\]
\[for \ i \in \set{1,2},\norm{R_i}_{1+\rho+ \sigma} \leq C(\norm{D\chi}_{\infty}) \norm{D\chi}_{\rho}\norm{u}_{1+\sigma}. \]
\item In Sobolev Spaces, for $s>\frac{d}{2}$ we get the following estimates
	\begin{itemize} 
	\item without the diffeomorphism hypothesis:
	\[  \norm{R_0}_{H^{1+\rho +min(1+\rho,s-\frac{d}{2})}} \leq C\norm{D\chi}_{\rho}\norm{u}_{H^{1+s}} \]
	\[ \norm{R_1}_{H^{1+\rho+s}} \leq C(\norm{D\chi}_{\infty}) \norm{D\chi}_{\rho}\norm{u}_{H^{1+s}}. \]
	\item Suppose moreover that $\chi$ is a diffeomorphism:
	\[ \norm{R_2}_{H^{1+\rho+s}} \leq C(\norm{D\chi}_{\infty},\norm{D\chi^{-1}}_{\infty}) \norm{D\chi}_{\rho}\norm{u}_{H^{1+s}}. \]
	\end{itemize}
The same estimates hold in the local spaces.
\end{itemize}
\end{theorem} 
As Alinhac remarked in \cite{Alinhac86}, a particular case of the previous theorem is Bony para-linearisation theorem but with the extra hypothesis of diffeomorphism, here it's is a full generalization because we dropped the diffeomorphism hypothesis. We find Bony's para-linearisation theorem when $\sigma=+\infty$, in this case only the term $T_{u'\circ \chi}\chi(x)$ appears and $\chi^\star u(x)$ is a part of the remainder. If on the other hand, $\chi \in C^\infty$, the term $T_{u'\circ \chi}\chi(x)$ becomes a part of the remainder and the paracomposition $\chi^\star u(x)$ coincides with the usual composition modulo a regularizing  operator. Thus Theorem \ref{paracomposition_section Paracomposition_subsec paracomp on the euclidean space_theorem paralinearisation of composition} appears as a linearisation theorem of $u \circ \chi$ as the sum of two terms, one depending mainly on the regularity of $u$ (and ``less" of $\chi$) and the other depending mainly on the regularity of $\chi$ (and ``less" of $u$).

\begin{remark} \label{paracomposition_section Paracomposition_subsec paracomp on the euclidean space_rem simplest example}
The simplest example for the paracomposition operator is when $\chi(x)=Ax$ is a linear operator and in that case we see that if $N$ is chosen sufficiently large in the definition:
\[u(Ax) = (Ax)^*u,\text{ and } T_{u'(Ax)}Ax = 0.\]
\end{remark}

\begin{remark} \label{paracomposition_section Paracomposition_subsec paracomp on the euclidean space_rem change of cutoff in paralinearisation theorem}
The proof of Theorem \ref{paracomposition_section Paracomposition_subsec paracomp on the euclidean space_theorem paralinearisation of composition} tell us that the if in the sum defining $\chi^\star$ we choose a different $N'\geq N$ then the operator is modified by a $\rho$ regularizing operator. 
 \end{remark}
\begin{theorem} \label{paracomposition_section Paracomposition_subsec paracomp on the euclidean space_theorem commutation para op and paracomp}
Consider $a \in \Gamma^m_\beta(\r^d)$, with $\beta \geq 0$, $\chi:\r^d \rightarrow \r^d$ a $C^{1+\rho}_{*,loc}$ map with $D\chi \in C^{\rho}_*$, $\rho>0$ and $1+\rho \notin \n$. Then there exists $q \in \Gamma^{m-\beta}_0(\r^d)$ such that we have the following formal symbolic calculus rule:
\begin{align*}
   \chi^\star T_a u &=op_{\omega_\chi}\bigg(\sigma_a(\chi(x),\xi)\frac{|\deter  D\chi(y)|}{\# \chi^{-1}(\chi(y))}\bigg) \chi^\star u+op_{\omega_\chi}\bigg(\sigma_q(\chi(x),\xi)\frac{|\deter  D\chi(y)|}{\# \chi^{-1}(\chi(y))}\bigg) \chi^\star u .
\end{align*}
\end{theorem}

To join Alinhac's work, the following proposition makes the link between his definition of the paracomposition operator in the case of a diffeomorphism and the one given here.

\begin{theorem} \label{paracomposition_section Paracomposition_subsec paracomp on the euclidean space_theorem link to Alinhac}
 Let $u$ be $W^{1,\infty}(\r^d)$ map and $\chi$ be a $C^{1+\rho}_{*,loc}$ diffeomorphism with $D\chi \in C^{\rho}_*$ and $\rho>0$.
Consider $\tilde{N}$ such that $2^{\tilde{N}}>sup_{k,\r^d} \abs{\Phi_k\chi'}^{-1}$ and $2^{\tilde{N}}>sup_{k,\r^d} \abs{\Phi_k\chi'}$. Put Alinhac's paracomposition operator:
 \[ \chi^* u=\sum_{k\geq 1}  \sum_{\substack{l\geq 0 \\ \abs{l-k} \leq \tilde{N}}} \phi_l(D) u_k\circ \chi  \]
 \[\text{then:} \ \chi^* u=\chi^\star u+R_3,\]
 Where the remainder verifies:
 \begin{itemize}
\item In Zygmund Spaces, for $\sigma>0$:
\[\norm{R_3}_{1+\rho+ \sigma} \leq C(\norm{D \chi^{-1}}_\infty,\norm{D\chi}_{\infty}) \norm{D\chi}_{\rho}\norm{u}_{1+\sigma}. \]
\item In Sobolev Spaces, for $s>\frac{d}{2}$:
\[ \norm{R_3}_{H^{1+\rho+s}} \leq C(\norm{D \chi^{-1}}_\infty,\norm{D\chi}_{\infty}) \norm{D\chi}_{\rho}\norm{u}_{H^{1+s}}. \]
The same estimates hold in the local spaces.\\
Take $a \in \Gamma^m_{\beta}(\r^d)$ and $q$ as in Theorem \ref{paracomposition_section Paracomposition_subsec paracomp on the euclidean space_theorem commutation para op and paracomp} then:
\begin{align*}
   \chi^\star T_a u &=T_{a^*} \chi^\star u+T_{q^*}\chi^\star u \\
   \chi^* T_a u &=T_{a^*} \chi^* u+T_{{q'}^*} \chi^* u \text{ with } q' \in \Gamma^{m-\beta}_{0}(\r^d).  \\
\end{align*}
\end{itemize}
\end{theorem}
\begin{remark} \label{paracomposition_section Paracomposition_subsec paracomp on the euclidean space_rem change of cutoff Alinhac}
As in remark \ref{paracomposition_section Paracomposition_subsec paracomp on the euclidean space_rem change of cutoff in paralinearisation theorem}, the proof of Theorem \ref{paracomposition_section Paracomposition_subsec paracomp on the euclidean space_theorem link to Alinhac} tell us that the if in the sum defining $\chi^*$ we choose a different $\tilde{N}'\geq \tilde{N}$ then the operator is modified by a $\rho$ regularizing operator. 
 \end{remark}
 \begin{remark} \label{paracomposition_section Paracomposition_subsec paracomp on the euclidean space_rem exact remainder in change of variable para}
As a corollary of Theorem \ref{paracomposition_section Paracomposition_subsec paracomp on the euclidean space_theorem link to Alinhac} we get that in Theorem \ref{paracomposition_section Pull-back of pseudo and para- differential operators_theorem change of variable para}:
\begin{equation}\label{reste changement de variable}
R\chi=T_{(T_au)'\circ \chi}\chi-T_{a^*}T_{u'\circ \chi}\chi.
\end{equation}
 \end{remark}
 \begin{remark}\label{paracomposition_section Paracomposition_subsec paracomp on the euclidean space_rem paracomposition from R to T}
 All of the result of this section extend naturally to the functions and operators defined on the Torus.
 \end{remark}
\subsection{Proofs}
We will give the proof for the estimates in global spaces, for local spaces it is sufficient to see that the given estimates hold under the hypothesis that all the functions used have a compact support and to pass to local spaces estimates it is sufficient to multiply by functions in $C^\infty_0$ which don't modify the estimates given (we don't make any boundary estimates).
\subsubsection*{Proof of Theorem \ref{paracomposition_section Paracomposition_subsec paracomp on the euclidean space_theorem defintion of paracomposition} and \ref{paracomposition_section Paracomposition_subsec paracomp on the euclidean space_theorem paralinearisation of composition}}
Take $\chi:\r^d \rightarrow \r^d$ be a $C^{1+\rho}_*$ map with $\rho>0$ put $\b=\b(0,N+1)$.\\
We start by the Zygmund spaces estimates (thus we don't suppose that $\chi$ is a diffeomorphism): 
\[\norm{\Phi_{k+N}u_k\circ \chi}_\infty \leq C \norm{u_k}_\infty  \leq 2^{-k\sigma}  \norm{u}_\sigma \]
and $\supp \ \Phi_{k+N}u_k\circ \chi \subset 2^k \b$.\\
Thus by Proposition \ref{paracomposition_Notations and functional analysis_proposition Zygmund spaces on balls}, for $\sigma >0$:
\[\chi^\star u \in C^\sigma_*(\r^d) \text{ and } \norm{\chi^\star u}_{\sigma} \leq \frac{C(N)}{1-2^{-\sigma}} \norm{ u}_{\sigma}.  \]
For Sobolev estimates we suppose that $\chi$ is a diffeomorphism and by the change of variables formula we have for $s>0$:
\[\norm{\Phi_{k+N}u_k\circ \chi}_{L^2} \leq C(\norm{D \chi^{-1}}_\infty) \norm{u_k}_{L^2}  \leq C (\norm{D \chi^{-1}}_\infty)2^{-ks}  \norm{u}_{H^s}\]
  and $ \supp \Phi_{k+N}u_k\circ \chi \subset 2^k \b$. \\
Thus by Proposition \ref{paracomposition_Notations and functional analysis_proposition Zygmund spaces on balls}, for $\sigma >0$:
\[\chi^\star u \in H^s(\r^d) \text{ and } \norm{\chi^\star u}_{H^s} \leq \frac{C(N,\norm{D \chi^{-1}}_\infty)}{1-2^{-s}} \norm{ u}_{H^s}.  \]
Now we compute the estimates on the remainders in the linearisation formula.
\[R_0=\sum_{k \geq 1} \bigg( \int_0^1 \Phi_{k-1}u'(\tau \Phi_k\chi+(1-\tau)\Phi_{k-1} \chi)- \Phi_{k-1}(u' \circ \chi)d\tau\bigg) \phi_k\chi=\sum_k r^0_k \chi_k, \]
and for $ r^0_k$ we compute
\begin{align*}
 &\int_0^1 \Phi_{k-1}u'(\tau \Phi_k\chi+(1-\tau)\Phi_{k-1} \chi)- \Phi_{k-1}(u' \circ \chi)d\tau \\
 &=\int_0^1 \Phi_{k-1}\left(u'\circ \left(\tau \Phi_k\chi+(1-\tau)\Phi_{k-1} \chi\right)- u' \circ \chi\right)d\tau \\
  &=\int_0^1 \Phi_{k-1}\left(\int_0^1 u''\circ \left(t\left(\tau \Phi_k\chi+(1-\tau)\Phi_{k-1} \chi\right)+(1-t)\chi\right)dt\left[\left(\tau \Phi_k\chi+(1-\tau)\Phi_{k-1} \chi\right)-\chi\right]\right)d\tau \\
   &=\int\limits_{[0,1]^2} \Phi_{k-1}\left( u''\circ \left(t\left(\tau \Phi_k\chi+(1-\tau)\Phi_{k-1} \chi\right)+(1-t)\chi\right)\left[\left(\tau (\Phi_k\chi-\chi)+(1-\tau)(\Phi_{k-1} \chi-\chi)\right)\right]\right)dt d\tau.
 \end{align*}
 Thus, by a standard paraproduct decomposition, if $\sigma <1$:
\[\norm{r^0_k}_\infty \leq C 2^{k(-\sigma-\rho)}, \]
and if $\sigma=1$:
\[\norm{r^0_k}_\infty \leq C k 2^{k(-1-\rho)}\leq C2^{-k}, \]
Which sums up in $\norm{r^0_k}_\infty \leq C2^{-min(1+\rho,\sigma)k}$ for $\sigma \leq 1$. By the same computations we have analogous estimates on $\norm{\partial^\alpha r^0_k}$ which gives the desired result for $\sigma>1$ and clearly $r^0_k \in C^\infty$ which gives the desired estimates on $R_0$ by Lemma \ref{paracomposition_Notations and functional analysis_lemme Meyer multiplier} and the fact that $r^0_0=0$, both in the Sobolev et Zygmund cases without the diffeomorphism hypothesis.\\
\begin{align*}
R_1&=\sum_{k\geq 0}\phi_{k+N}(D)[u_k\circ\Phi_{k} \chi-u_k\circ \chi]\\
&=\sum_{k\geq 0}\phi_{k+N}(D)\left[\int_0^1 u_k'(t\Phi_{k}+(1-t)\chi)dt)(\Phi_{k}\chi - \chi)\right]\\
&=\sum_{k\geq 0}\phi_{k+N}(D)[r^1_k(\Phi_{k}\chi - \chi)].
\end{align*}
We have:
\[\norm{r^1_k}_\infty \leq C 2^{-k\sigma} \]
combining this with Propositions \ref{paracomposition_Notations and functional analysis_proposition Zygmund spaces on balls}, \ref{paracomposition_Notations and functional analysis_proposition Sobolev spaces on balls} and the fact that $r^1_0=0$ we get the desired estimates again in both in the Sobolev et Zygmund cases without the diffeomorphism hypothesis.\\
The proof of the estimates on $R_2$ relies on oscillatory integral techniques that come from Lemma \ref{paracomposition_Notions of microlocal analysis_Fourier Integral Operators_lemme fonda integral oscillante} . For the sake of completion we will give the explicit computations without directly using the lemma.\\
\begin{align*}
R_2(x)&=\sum_k(Id-\Phi_{k+N})(D)u_k\circ \Phi_k \chi(x).
\end{align*}
We will prove that for $j \geq k+N+1, \nu\geq \rho>0$, we have:
\begin{equation} \label{paracomposition_section Paracomposition_subsec paracomp on the euclidean space_subsubsection proof eq1}
\norm{\phi_j(D)u_k \circ \Phi_k \chi}_\infty \leq C_\nu(\norm{D\chi}_{\rho})2^{-j\nu} 2^{k(\nu-\rho)}\norm{u_k}_\infty
\end{equation}
which will be sufficient to give the Zygmund estimates on $R_2$ because we will have:
\begin{align*}
\norm{\phi_j(D)R_2}_\infty&\leq \sum_{\substack{k\geq 0 \\ k\leq N-j+1}}\norm{\phi_j(D)u_k \circ \Phi_k \chi}_\infty \\
 						&\leq\sum_{\substack{k\geq 0 \\ k\leq N-j+1}} C_\nu(\norm{D\chi}_{\rho})2^{-j\nu} 2^{k(\nu-\rho)}\norm{u_k}_\infty  \\
					&\leq \sum_{\substack{k\geq 0 \\ k\leq N-j+1}} C_\nu(\norm{D\chi}_{\rho})2^{-j\nu} 2^{k(\nu-\rho)}\norm{u_k}_\infty \\
					 & \leq \sum_{\substack{k\geq 0 \\ k\leq N-j+1}} C_\nu(\norm{D\chi}_{\rho})2^{-j\nu} 2^{k(\nu-\rho-\sigma-1)}\norm{u}_{1+\sigma},
\end{align*}
Taking $\nu>1+\rho+\sigma$ we dominate the last expression by:
\[ C_\nu(\norm{D\chi}_{\rho})2^{-j(\rho+\sigma+1)}\norm{u}_{1+\sigma}\]
which gives the desired Zygmund estimate.\\

For the Sobolev estimates we will prove that:
\begin{equation} \label{paracomposition_section Paracomposition_subsec paracomp on the euclidean space_subsubsection proof eq2}
\norm{\phi_j(D)u_k \circ \Phi_k \chi}_2 \leq C_\nu(\norm{D\chi}_{\rho})2^{-j\nu} 2^{k(\nu-\rho)}\norm{u_k \circ \Phi_k \chi}_2,
\end{equation}
which then necessitates the diffeomorphism hypothesis on $\chi$ to have:
\begin{equation*}
\norm{\phi_j(D)u_k \circ \Phi_k \chi}_2 \leq C_\nu(\norm{D\chi}_{\rho},\norm{D\chi^{-1}}_\infty)2^{-j\nu} 2^{k(\nu-\rho)}\norm{u_k }_2,
\end{equation*}
And the desired estimates follow exactly as in the Zygmund case.\\
Now we prove \eqref{paracomposition_section Paracomposition_subsec paracomp on the euclidean space_subsubsection proof eq1} and \eqref{paracomposition_section Paracomposition_subsec paracomp on the euclidean space_subsubsection proof eq2}, we will follow Taylor's proof of Alinhac's lemma given in Appendix B of Chapter 2 of  \cite{Taylor07}.  To make the desired estimates we will put in a test function $f \in C^\infty_b$ as it's usually done with oscillatory integral estimates:

\begin{align} \label{paracomposition_section Paracomposition_subsec paracomp on the euclidean space_subsubsection proof eq3}
\phi_j(D)f u_k \circ \Phi_k \chi(x)&=\int e^{i(x-y).\xi}\phi_j(\xi)\phi_k(\eta)f(y)\hat{u_k}(\eta)e^{i \Phi_k \chi(y).\eta }d\eta dy d\xi 
\end{align}
Set 
\[\omega_k(y,\eta,\xi)=\Phi_k \chi(y).\eta-y.\xi, \]
\[L_k(y,\eta,\xi,\partial_y)=\frac{\Phi_k \chi'(y)^t.\eta-y.\xi}{i\abs{\Phi_k \chi'(y)^t.\eta-y.\xi}^2}.\nabla_y.\]
Given the definition of N we have:
\[\abs{\Phi_k \chi'(y)^t.\eta-y.\xi}\geq C(\abs{\eta}+\abs{\xi}) \text{ on } \supp \ \phi_j(\xi)\phi_k(\eta), \]
Thus $L_k$ is well defined and regular, moreover $L_k e^{i\omega_k}=e^{i\omega_k}$. Integrating by parts in \eqref{paracomposition_section Paracomposition_subsec paracomp on the euclidean space_subsubsection proof eq2}:
\begin{align*}
\phi_j(D)f u_k \circ \Phi_k \chi(x)&=\int e^{ix.\xi}\phi_j(\xi)\phi_k(\eta)\hat{u_k}(\eta)e^{i w_k } (L_k^t)^\nu f(y)d\eta dy d\xi. 
\end{align*}
Note that $(L_k^t)^\nu f$ is homogeneous with degree $-\nu$ in $(\eta,\xi)$, and smooth on the support of $\phi_j(\xi)\phi_k(\eta)$. Also 
\begin{align} \label{paracomposition_section Paracomposition_subsec paracomp on the euclidean space_subsubsection proof eq4}
\abs{(L_k^t)^\nu f(y)}\leq  C(\norm{f}_\nu,\norm{D\chi}_{\rho}) 2^{\nu-\sigma} \text{ on } \abs{\xi}^2+\abs{\eta}^2=1. 
\end{align}
Next on a box containing $\supp \ \phi_j(\xi)\phi_k(\eta)$, write
\[(L_k^t)^\nu f(y)=\sum_{(\alpha,\beta) \in \Lambda} a_{k \nu \alpha \beta}(y)e^{i\alpha.\xi+i\beta.\eta}=2^{-j \nu}\sum_{(\alpha,\beta) \in \Lambda} a_{k \nu \alpha \beta}(y)e^{i2^{-j}\alpha.\xi+i2^{-j}\beta.\eta},\]
where $\Lambda$ is an appropriate lattice and 
\begin{align} \label{paracomposition_section Paracomposition_subsec paracomp on the euclidean space_subsubsection proof eq5}
\sum_{(\alpha,\beta) \in \Lambda} \norm{a_{k \nu \alpha \beta}}_\infty \leq  C(\norm{f}_\nu,\norm{D\chi}_{\rho}) 2^{\nu-\sigma}.
\end{align}
So \eqref{paracomposition_section Paracomposition_subsec paracomp on the euclidean space_subsubsection proof eq3} becomes for $j\geq1$:
\begin{align} \label{paracomposition_section Paracomposition_subsec paracomp on the euclidean space_subsubsection proof eq6}
&\phi_j(D)f u_k \circ \Phi_k \chi(x)\\
&=2^{-j \nu}\sum_{(\alpha,\beta) \in \Lambda} \int e^{ix.\xi}\phi_j(\xi)\phi_k(\eta)\hat{u_k}(\eta)e^{i w_k }  a_{k \nu \alpha \beta}(y)e^{i2^{-j}\alpha.\xi+i2^{-j}\beta.\eta}d\eta dy d\xi \nonumber \\
						&= 2^{-j \nu}\sum_{(\alpha,\beta) \in \Lambda} \int e^{i(x-y).\xi}\phi_j(\xi)u_k(\Phi_k \chi(y)+2^{-j}\beta)  a_{k \nu \alpha \beta}(y)e^{i2^{-j}\alpha.\xi} dy d\xi \nonumber \\
						&= 2^{-j \nu}\sum_{(\alpha,\beta) \in \Lambda} \int e^{i(x-y).\xi}2^{jn}\hat{\phi_1}(2^{j}(x-y)+\alpha)u_k(\Phi_k \chi(y)+2^{-j}\beta)  a_{k \nu \alpha \beta}(y)dy .\nonumber \\
						&= 2^{-j \nu}\sum_{(\alpha,\beta) \in \Lambda} (a_{k \nu \alpha \beta}\cdot u_k(\Phi_k \chi+2^{-j}\beta)) \ast g_\alpha(x),
\end{align}
Where $g_\alpha(x)=2^{jn}\hat{\phi_1}(2^{j}x+\alpha)$ thus
\begin{align} \label{paracomposition_section Paracomposition_subsec paracomp on the euclidean space_subsubsection proof eq7}
\norm{g_\alpha}_{L^1}&=2^{jn}\int\abs{\hat{\phi_1}(2^{j}x+\alpha)}dx=\norm{\hat{\phi_1}}_{L^1}.
\end{align}
For $j=0$ we have an analog inequality.\\
Using the classic Young and H{\"o}lder inequalities combined with \eqref{paracomposition_section Paracomposition_subsec paracomp on the euclidean space_subsubsection proof eq5}, \eqref{paracomposition_section Paracomposition_subsec paracomp on the euclidean space_subsubsection proof eq7} and taking $f \rightarrow 1$ gives us \eqref{paracomposition_section Paracomposition_subsec paracomp on the euclidean space_subsubsection proof eq1} and \eqref{paracomposition_section Paracomposition_subsec paracomp on the euclidean space_subsubsection proof eq2}. This concludes the proof.

\subsubsection*{Proof of Theorem \ref{paracomposition_section Paracomposition_subsec paracomp on the euclidean space_theorem commutation para op and paracomp}}
Take $a \in \Gamma^m_\beta(\r^d)$, with $\beta \geq 0$ and $\chi:\r^d \rightarrow \r^d$ a $C^{1+\rho}_*$ map with $\rho>0$. We compute:
\begin{equation}\label{paracomposition_section Paracomposition_subsec paracomp on the euclidean space_subsubsection proof eq8}
\chi^\star T_a u=\sum_{k\geq0}\Phi_{k+N}[(T_au)_k \circ \chi],
\end{equation}
Note that $(T_au)_k$ can been as $T_{\phi_k}T_au$ and seeing this a modification of the cut-off function by Proposition \ref{paracomposition_Notions of microlocal analysis_Paradifferential Calculus_definition regularisation of a symbol by cutoff} we get:
\begin{align*}
(T_au)_k&=T_{\phi_k}T_au=T_{a}T_{\phi_k}u+T_{q^k}u ,\text{ with } q^k \in \Gamma^{m-\beta}_0(\r^d).
\end{align*}
Put $q=\sum q^k$
 then \eqref{paracomposition_section Paracomposition_subsec paracomp on the euclidean space_subsubsection proof eq8} becomes:
\begin{align*}
\chi^\star T_a u=\sum_{k\geq0}\Phi_{k+N}[(T_a u_k) \circ \chi] +\sum_{k\geq0}\Phi_{k+N}[(T_{q^k} u_k) \circ \chi].
\end{align*}
And the formal discussion and computations in part \ref{paracomposition_section Pull-back of pseudo and para- differential operators} give the desired result.
\subsubsection*{Proof of Theorem \ref{paracomposition_section Paracomposition_subsec paracomp on the euclidean space_theorem link to Alinhac}}
The only thing left to prove is the estimate on $R_3$.\\
\[R_3=\sum_k\underbrace{\Phi_{k-\tilde{N}}(D)u_k\circ \Phi_k \chi(x)}_1 + \phi_N(D) u_k\circ \chi\]
 $ \phi_N(D) u_k\circ \chi$ is $C^\infty$ so we only have to the estimate to the first term on the left hand side.
Estimating 1 is exactly as \eqref{paracomposition_section Paracomposition_subsec paracomp on the euclidean space_subsubsection proof eq3} but with $\phi_j$ substituted by $\Phi_{k-\tilde{N}}$. The core of the estimation relies on the fact that $L_k$ should be well defined and regular on $ \supp \Phi_{k-\tilde{N}}(\xi) \phi_k(\eta)$ which is the case give our choice of $\tilde{N}$ and the fact $k\geq1$. We also have the estimate:
\[\abs{\Phi_k \chi'(y)^t.\eta-y.\xi}\geq C(\abs{\eta}+\abs{\xi}) \text{ on } \supp \Phi_{k-\tilde{N}}(\xi)\phi_k(\eta).\]
The proof than exactly follows as for $R_2$.

\subsection{Main results for paracomposition on open subsets}
The previous definition of the operator $\chi^\star$ on functions defined on $\r^d$ relied heavily on the Littelwood-Paley theory which doesn't make it immediately extendable to the open domain case. In \cite{Alinhac86}, Alinhac was able to define such an operator profiting from the continuity of $\chi^*$ on the local function spaces and a partition of unity on the open domains. More precisely consider $(V_i,\Theta_i)$ a partition of unity locally finite of $\Omega'$ then:
\[u\circ \chi=\sum_i \Theta_i u \circ \chi\]
where $\Theta_i u$ is seen as a function of $\r^n$ with the natural extension by 0. In order to have the same natural extension for $\chi$, $$\chi^{-1}(\supp \Theta_i)$$ needs to be compact we thus have to suppose that $\chi$ is a proper map\footnote{Note that this extra hypothesis is needed for the methods used to work and is not intrinsic to the problem. Also this hypothesis is immediately verified in the diffeomorphism case treated by Alinhac.}. Under this hypothesis consider $\zeta_i\in C^\infty_0(\Omega)$ such that $\zeta_i=1$ on $\chi^{-1}(\supp \Theta_i)$:
\begin{equation}\label{paracomposition_section Paracomposition_subsec paracomp on open sets eq 1}
u\circ \chi=\sum_i \zeta_i \Theta_i u \circ \zeta_i \chi,
\end{equation}
where $\zeta_i \chi$ is seen as a function of $\r^n$ with the natural extension by $0$. Henceforth we define:
\[
\chi^\star u=\sum_i \zeta_i \cdot (\zeta_i \chi)^{\star}\Theta_i u.
\]

 \begin{theorem} \label{paracomposition_section Paracomposition_subsec paracomp on open sets_theorem definition of paracomposition}
 Let $\chi:\Omega \rightarrow \Omega'$ be a $C^{1+\rho}_{*,loc}$ proper map with $D\chi \in C^\rho_*$ and $\rho>0$. Consider $(V_i,\Theta_i)$ a partition of unity locally finite of $\Omega'$ and $\zeta_i$ the associated functions as previously.  Then for all $\sigma,s \in \r^*_+$ the following maps are continuously extended: 
\[
   \chi^\star: C^\sigma_{*}(\Omega') \rightarrow  C^\sigma_{*}(\Omega) \text{ and } \chi^\star:C^\sigma_{*,loc}(\Omega') \rightarrow  C^\sigma_{*,loc}(\Omega),
\]
  if moreover $\chi$ is a diffeomorphism then we have the Sobolev estimates:  
   \[ \chi^\star:H^s(\Omega') \rightarrow H^s(\Omega) \text{ and }  \chi^\star:H^s_{loc}(\Omega') \rightarrow H^s_{loc}(\Omega),
\] 
 where $\Theta_iu$ and $\zeta_i \chi$ are treated as functions on $\r^d$. And
In the sum defining each $(\zeta_i \chi)^{\star}$ a choice $$N_i,2^{N_i}\geq\sup_{\supp \Theta_i}\chi'$$ is made by the definition in section \ref{paracomposition_section Paracomposition_subsec paracomp on the euclidean space}, but by remark \ref{paracomposition_section Paracomposition_subsec paracomp on the euclidean space_rem change of cutoff in paralinearisation theorem} in order to simplify the computations we can take the same $$N\geq N_i,2^{N}\geq\sup_{\Omega}\chi'$$ uniformly for all the operators and this modifies the definition by a $\rho$ regularizing operator. \\
Making a different choice $(V'_i,\Theta'_i,\zeta'_i)$, which gives a different operator $\chi^{\star}_1$ then
\[ \forall u \in C^{\sigma}_{*}(\r^d)\cup C^{\sigma}_{*,loc}(\r^d) , \chi^\star u= \chi^{\star}_1 u +R'u.\]   
with $R'u \in C^\infty$.\\

Consider $\tilde{\chi}:\Omega' \rightarrow \Omega''$ a a $C^{1+\tilde{\rho}}_{*,loc}$ proper map with $D\tilde{\chi} \in C^{\tilde{\rho}}_*$ with $\tilde{\rho}>0$, then the previous operation has the natural fonctorial property:
\[ \forall u \in C^{\sigma}_{*}(\Omega'')\cup C^{\sigma}_{*,loc}(\Omega'') , \chi^\star \tilde{\chi}^\star u= ({\chi \circ \tilde{\chi}})^\star u +\tilde{R}u.\]   
 \[\text{with } \tilde{R}, \  \tilde{R}:C^{\sigma}_{*}(\Omega'') \rightarrow  C^{\sigma+min(\rho,\tilde{\rho})}_{*}(\Omega), \ \tilde{R}:C^{\sigma}_{*,loc}(\Omega'') \rightarrow  C^{\sigma+min(\rho,\tilde{\rho})}_{*,loc}(\Omega),\]
 and if $\chi$ and $\tilde{\chi}$ are diffeomorphisms: 
 \[  \tilde{R}:H^{s}(\Omega'') \rightarrow H^{s+min(\rho,\tilde{\rho})}(\Omega), \ \tilde{R}:H^{s}_{loc}(\Omega'') \rightarrow H^{s+min(\rho,\tilde{\rho})}_{loc}(\Omega).\]
  \end{theorem}
   
 \begin{theorem}  \label{paracomposition_section Paracomposition_subsec paracomp on open sets_theorem paralinearisation of composition}
 Let u be a $W^{1,\infty}(\Omega)$ map and $\chi$ be a be a $C^{1+\rho}_{*,loc}$ proper map with $D\chi \in C^\rho_*$ and $\rho>0$. Then:
 \[u \circ \chi(x)=\chi^\star u(x)+ T_{u'\circ \chi}\chi(x)+ R_0(x)+R_1(x)+R_2(x)\]
 where the paracomposition given in the previous theorem verifies the estimates:
 \[\forall \sigma > 0, \norm{\chi^\star u(x)}_{\sigma}\leq C(\norm{D\chi}_{\infty})\norm{u(x)}_{\sigma},\]
 \[u'\circ \chi \in  \Gamma^0_{W^{0,\infty}(\Omega)}(\r^d) \text{ for }  u \text{  Lipchitz.}   \]
 The remainders are given by:
 \[R_0=\sum_{i}\sum_{k \geq 1} \zeta_i\bigg(\int_0^1 \Phi_{k-1}\Theta_i u'(\tau \Phi_k\zeta_i \chi+(1-\tau)\Phi_{k-1} \zeta_i \chi)- \Phi_{k-1}(\Theta_i u' \circ \zeta_i \chi)d\tau\bigg) \phi_k\zeta_i \chi,\]
 \[R_1=\sum_{i}\sum_{k\geq 0} \sum_{\substack{l\geq 0 \\ l\leq k+N}}\zeta_i(\phi_l(D)[\Theta_i u_k\circ\Phi_{k} \zeta_i \chi-\Theta_i u_k\circ \zeta_i \chi]),\]
 \[R_2=\sum_{i}\sum_k\zeta_i((Id-\Phi_{k+N})(D)\Theta_i u_k\circ \Phi_k \zeta_i \chi),\]
and the remainders verify the estimates: 
\begin{itemize}
\item In Zygmund Spaces, for $\sigma>0$:
\[\norm{R_0}_{1+\rho +min(1+\rho,\sigma)} \leq C \norm{D\chi}_{\rho}\norm{u}_{1+\sigma}\]
\[for \ i \in \set{1,2},\norm{R_i}_{1+\rho+ \sigma} \leq C(\norm{D\chi}_{\infty}) \norm{D\chi}_{\rho}\norm{u}_{1+\sigma}. \]
\item In Sobolev Spaces, for $s>\frac{d}{2}$ we get the following estimates
	\begin{itemize} 
	\item without the diffeomorphism hypothesis:
	\[  \norm{R_0}_{H^{1+\rho +min(1+\rho,s-\frac{d}{2})}} \leq C\norm{D\chi}_{\rho}\norm{u}_{H^{1+s}} \]
	\[ \norm{R_1}_{H^{1+\rho+s}} \leq C(\norm{D\chi}_{\infty}) \norm{D\chi}_{\rho}\norm{u}_{H^{1+s}}. \]
	\item Suppose moreover that $\chi$ is a diffeomorphism:
	\[ \norm{R_2}_{H^{1+\rho+s}} \leq C(\norm{D\chi}_{\infty},\norm{D\chi^{-1}}_{\infty}) \norm{D\chi}_{\rho}\norm{u}_{H^{1+s}}. \]
	\end{itemize}
The same estimates hold in the local spaces.
\end{itemize}
\end{theorem}

\begin{theorem} \label{paracomposition_section Paracomposition_subsec paracomp on open sets_theorem commutation between para op and paracomp}
Consider $a \in \Gamma^m_\beta(\r^d)$, with $\beta \geq 0$ and $\chi:\Omega \rightarrow \Omega'$ a $C^{1+\rho}_{*,loc}$ proper map with $D\chi \in C^{\rho}_*$, $\rho>0$ and $1+\rho \notin \n$. Then there exists $q \in \Gamma^{m-\beta}_0(\r^d)$ such that we have the following formal symbolic calculus rule:
\begin{align*}
   \chi^\star T_a u &=op_{\omega_\chi}\bigg(\sigma_a(\chi(x),\xi)\frac{|\deter  D\chi(y)|}{\# \chi^{-1}(\chi(y))}\bigg) \chi^\star u+op_{\omega_\chi}\bigg(\sigma_q(\chi(x),\xi)\frac{|\deter  D\chi(y)|}{\# \chi^{-1}(\chi(y))}\bigg) \chi^\star u .
\end{align*}
\end{theorem}

Again to join Alinhac's work:

\begin{theorem} \label{paracomposition_section Paracomposition_subsec paracomp on open sets_theorem link with Alinhac's paracom}
 Let u be $W^{1,\infty}(\Omega)$ map and  $\chi$ be a $W^{1,\infty}$ diffeomorphism, a $C^{1+\rho}_{*,loc}$ proper map with $D\chi \in C^\rho_*$ and $\rho>0$.
Again, consider $(V_i,\Theta_i)$ a partition of unity locally finite of $\Omega'$ and $\zeta_i$ the associated functions as previously.  Put Alinhac's paracomposition operator:
 \[  \chi^* u=\sum_i \zeta_i (\zeta_i \chi)^{*}\Theta_i u \text{ then} :\]
 \[ \chi^* u=\chi^\star u+R_3,\]
 Where the remainder verifies:
 \begin{itemize}
\item In Zygmund Spaces, for $\sigma>0$:
\[\norm{R_3}_{1+\rho+ \sigma} \leq C(\norm{D \chi^{-1}}_\infty,\norm{D\chi}_{\infty}) \norm{D\chi}_{\rho}\norm{u}_{1+\sigma}. \]
\item In Sobolev Spaces, for $s>\frac{d}{2}$:
\[ \norm{R_3}_{H^{1+\rho+s}} \leq C(\norm{D \chi^{-1}}_\infty,\norm{D\chi}_{\infty}) \norm{D\chi}_{\rho}\norm{u}_{H^{1+s}}. \]
The same estimates hold in the local spaces.\\
Consider $a \in \Gamma^m_{\beta}(\r^d)$ and $q$ as in Theorem \ref{paracomposition_section Paracomposition_subsec paracomp on the euclidean space_theorem commutation para op and paracomp} then:
\begin{align*}
   \chi^\star T_a u &=T_{a^*} \chi^\star u+T_{q^*}\chi^\star u  \\
   \chi^* T_a u &=T_{a^*} \chi^* u+T_{{q'}^*} \chi^* u \text{ with } {q'} \in \Gamma^{m-\beta}_{0}(\r^d).  \\
\end{align*}
\end{itemize}
Again we have the same ``independence" of the definition of the operator $\chi^*$ (modulo a more regular term) with respect to the arbitrary choices made, more precisely, making a different choice $(V'_i,\Theta'_i,\zeta_i')$ which gives a different operator $\chi^{*}_1$ then
\[ \forall u \in C^{\sigma}_{*}(\r^d)\cup C^{\sigma}_{*,loc}(\r^d) , \chi^* u= \chi^{*}_1 u +R'u.\]   
with $R'u \in C^\infty$.\\
\end{theorem}

\subsection{Proof}

All of the estimates given come directly for the theorems of section \ref{paracomposition_section Paracomposition_subsec paracomp on the euclidean space}. The linearisation formulas come from Equation \eqref{paracomposition_section Paracomposition_subsec paracomp on open sets eq 1} and the linearisation theorems in section \ref{paracomposition_section Paracomposition_subsec paracomp on the euclidean space}. The only thing left to prove is the independency result with respect to the choice of $(V_i,\Theta_i,\zeta_i)$.
We start by the following lemma:
\begin{lemma} \label{paracomposition_section Paracomposition_subsec paracomp on open sets_proof lemma change of cutt-off}
Let $(\Theta,\zeta,\tilde{\zeta})\in C^{\infty}_0(\Omega')$ be such that $\zeta=1$ on $ \chi^{-1}(\supp \Theta)$ and $\tilde{\zeta}=1$ on $\supp  \zeta$ then:
\[\sum_{k\geq0}\zeta \Phi_{k+N}(D)[( \Theta u)_k\circ \zeta \chi]=\sum_{k\geq0}\tilde{\zeta} \Phi_{k+N}(D)[( \Theta u)_k\circ \tilde{\zeta} \chi]+F, \ F \in C^\infty\]
\end{lemma}
\begin{proof} 
Take $\Theta'\in C^{\infty}_0(\Omega')$ such that $\Theta'\circ \chi=0$ on $\supp \ \zeta$ and $\Theta'\circ \chi=1$ on $\supp \ \tilde{\zeta}- \zeta$ and compute:
\begin{align*}
&\sum_{k\geq0}\zeta \Phi_{k+N}(D)[( \Theta u)_k\circ \zeta \chi]\\
&=\sum_{k\geq0}\tilde{\zeta} \Phi_{k+N}(D)[( \Theta u)_k\circ \tilde{\zeta} \chi]+\sum_{k\geq0}(\zeta-\tilde{\zeta}) \Phi_{k+N}(D)[(\Theta u)_k\circ \tilde{\zeta} \chi] \\
&=\sum_{k\geq0}\tilde{\zeta} \Phi_{k+N}(D)[( \Theta u)_k\circ \tilde{\zeta} \chi]+\underbrace{\sum_{k\geq0}(\zeta-\tilde{\zeta}) \Phi_{k+N}(D)[(\Theta'( \Theta u)_k)\circ \tilde{\zeta} \chi]}_F.
\end{align*}
And we have by integration by parts, $\forall l\in \n$:
\[ \Theta' (\Theta u)_k=2^{-kl}\int \frac{e^{i(x'-y')\xi}}{i(x'-y')^l}\Theta'(x)\Theta(y)\phi_1(2^{-k}\xi)u(y) dy d\xi,\]
\[\text{thus}, \ \norm{\Theta' (\Theta u)_k}_\infty\leq C_l 2^{-k(l-n)}, \text{ and } F \in C^\infty.\]
\end{proof}
Given (i,j) such that $\supp \Theta_i \cap \supp \Theta'_j \ne \emptyset$ we define $\tilde{\zeta}_{i,j} \in C^{\infty}_0(\Omega)$ such that $\tilde{\zeta}_{i,j}=1$ on $ \supp \zeta_i \cup \supp \zeta'_j$.
\begin{align*}
\chi^\star u&=\sum_i \zeta_i \cdot (\zeta_i \chi)^{\star}\Theta_i u=\sum_{k\geq0}\sum_{i,j} \zeta_i \Phi_{k+N}(D)[( \Theta_i \Theta'_j u)_k\circ \zeta_i\chi] \\
		&=\sum_{k\geq0}\sum_{i,j} \tilde{\zeta}_{i,j}\Phi_{k+N}(D)[( \Theta_i \Theta'_j u)_k\circ \tilde{\zeta}_{i,j}\chi] +F, \ F \in C^\infty \\
		&=\sum_{k\geq0}\sum_{i,j} \zeta'_j\Phi_{k+N}(D)[( \Theta_i \Theta'_j u)_k\circ \zeta'_j\chi]+F+F', \ F' \in C^\infty \\
		&=\chi_1^\star u +F+F',
\end{align*}
which gives the desired result and ends the proof.\\

\end{document}